\documentclass[12pt,english]{amsart}
\usepackage[cp1251]{inputenc}
\usepackage[english]{babel}
\usepackage{amsmath,amsthm,amssymb}
\newtheorem{theorem}{Theorem}[section]
\newtheorem{lemma}[theorem]{Lemma}
\newtheorem{corollary}[theorem]{Corollary}
\newtheorem{remark}[theorem]{Remark}

\numberwithin{equation}{section}
\theoremstyle{definition}
\newtheorem{definition}[theorem]{Definition}

\begin{document}

\title[]
{Marcinkiewicz-type discretization
of $L^p$-norms under the Nikolskii-type inequality assumption}

\author{Egor Kosov}

\maketitle

\begin{abstract}
The paper studies the sampling discretization problem
for integral norms on subspaces of $L^p(\mu)$.
Several close to optimal results are obtained on subspaces
for which certain Nikolskii-type inequality is valid.
The problem of norms discretization is connected with the probabilistic
question about the approximation with high probability
of marginals of a high dimensional
random vector by sampling. As a byproduct of our approach
we refine the result of  O. Gu$\acute{e}$don and M.~Rudelson
concerning the approximation of marginals.
In particular, the obtained improvement
recovers a theorem
of J. Bourgain, J. Lindenstrauss, and V.~Milman
concerning embeddings
of  finite dimensional subspaces of $L^p[0, 1]$ into $\ell_p^m$.
The proofs in the paper use the recent
developments of the chaining technique
by R.~van~Handel.
\end{abstract}



\section{Introduction}

Let $\Omega$ be a compact set endowed with some probability Borel measure $\mu$.
Let $L$ be an $N$-dimensional subspace of $L^p(\mu)\cap C(\Omega)$.
In this paper we consider the following
problem of sampling discretization. Let $C>c>0$ be fixed.
What is the least possible number $m$
of points $X_1,\ldots, X_m$
such that
$$
c\|f\|_p^p\le \frac{1}{m}\sum\limits_{j=1}^m|f(X_j)|^p\le C\|f\|_p^p
$$
for every $f\in L$?
Here
$\displaystyle
\|f\|_p:=\Bigl(\int |f|^p\, d\mu \Bigr)^{1/p},\quad \|f\|_\infty:=\sup\limits_{x\in \Omega}|f(x)|.$

The obvious bound is $m\ge N$, so we are seeking for the conditions
on the subspace $L$ under which the sampling discretization problem could be solved with
the number of points $m$ close to the dimension of the subspace
(ideally, with $m=O(N)$).
This and similar
problems have been extensively
studied in recent years (see \cite{5-1}, \cite{5-2}, \cite{DPTT},
\cite{KT}, \cite{Tem1}, \cite{Tem2}, \cite{Tem3}, and \cite{Tem4}).
The first classical result
of such type was obtained in the 1930s by Marcinkiewicz and Marcinkiewicz-Zygmund
for discretization of the $L^p$-norms of the univariate trigonometric
polynomials
(see \cite{Zyg} or \cite[Theorem 1.3.6]{TemBook}).
That is why the described above problem of sampling discretization is also called
the
Marcinkiewicz-type discretization problem
(see \cite{Tem1} and \cite{Tem2}, where
this notion was introduced).

In this paper we take the probabilistic approach and assume that points
$X_1, \ldots, X_n$ are chosen randomly and independently
and distributed according to the measure $\mu$.
For any $B\subset L$ let
\begin{equation}\label{V-funct}
V_p(B):=\sup\limits_{f\in B}
\Bigl|\frac{1}{m}\sum\limits_{j=1}^m|f(X_j)|^p - \|f\|_p^p\Bigr|.
\end{equation}
If $B=B_p(L):=\{f\in L\colon \|f\|_p\le1\}$ and
one can show that for some $\varepsilon\in(0,1)$ and some number $m$
the bound
$V_p(B_p(L))\le \varepsilon$ holds with positive probability,
then one has
$$
(1-\varepsilon)\|f\|_p^p\le
\frac{1}{m}\sum\limits_{j=1}^m|f(X_j)|^p
\le (1+\varepsilon)\|f\|_p^p\quad
\forall f\in L.
$$
Note that by Chebyshev's inequality
$P\bigl(V_p(B)\le 2\mathbb{E}\bigl[V_p(B)\bigr]\bigr)\ge 2^{-1}$,
thus it is sufficient to provide good bounds for the expectation
$\mathbb{E}\bigl[V_p(B)\bigr]$. Here and further $\mathbb{E}$ denotes the expectation
of a random variable.

We note that in this formulation the problem is equivalent to the following
problem of approximation of one-dimensional marginals
of a random vector $\mathbf{u}$ by sampling. Let $\mathbf{u}$ be a random vector
in $\mathbb{R}^N$ endowed with some inner product $\langle\cdot, \cdot\rangle$
and let $K\subset \mathbb{R}^N$.
The problem is to understand how well one can approximate one-dimensional marginals of
$\mathbf{u}$ by sampling with high probability,
i.e. let $\mathbf{u}^1, \ldots, \mathbf{u}^m$
be $m$ independent copies of the vector $\mathbf{u}$ and
let
$$
U_p(K):=\sup\limits_{y\in K}
\Bigl|\frac{1}{m}\sum\limits_{j=1}^m|\langle y, \mathbf{u}^j\rangle|^p -
\mathbb{E}|\langle y, \mathbf{u}\rangle|^p\Bigr|.
$$
How many independent copies of $\mathbf{u}$ are needed to guarantee $U_p(K)\le\varepsilon$
with high probability?

On the one hand, for any fixed set $K\subset \mathbb{R}^N$
one can consider the set of functions
$$B:=\{f_y(\cdot)=\langle y,\cdot\rangle\colon y\in K\}\subset L^p(\mu),$$
where $\mu$ is the distribution of $\mathbf{u}$,
and obtain the equality $\mathbb{E}\bigl[V_p(B)\bigr] = \mathbb{E}\bigl[U_p(K)\bigr]$.
On the other hand, for any fixed inner product $\langle\cdot, \cdot \rangle$
on an $N$-dimensional subspace $L\subset L^p(\mu)$ and for any $B\subset L$
one can take the orthonormal basis $u_1, \ldots, u_N$
of this space $L$ with respect to this inner product
and consider i.i.d. random vectors 
$\mathbf{u}^j:=\bigl(u_1(X_j),\ldots, u_N(X_j)\bigr)$ in $\mathbb{R}^N$.
If one now take $K:=\{y=(y_1,\ldots, y_N)\in \mathbb{R}^N\colon y_1u_1+\ldots+y_Nu_N\in B\}$,
then $\mathbb{E}\bigl[U_p(K)\bigr]=\mathbb{E}\bigl[V_p(B)\bigr]$.
This problem of approximation of marginals has also been
extensively studied
(see \cite{ALPT-J}, \cite{GM}, \cite{GueRud}, \cite{Rud1}, \cite{Rud2}, \cite{Tikh},
\cite{Versh}, \cite{Versh2} and citations therein).

We note that the probabilistic approach may not provide the optimal result for the initial
problem of sampling discretization.
For example, this is the case when $p=2$. In recent paper \cite{LimTem},
the famous result of A.~Marcus, D.A. Spielman, N. Srivastava
from \cite{MSS} has been combined with the iteration procedure from \cite{NOU}
to show the following assertion.
There are positive constants $C_1, C_2, C_3$ such that
for any subspace $L\subset L^2(\mu)$,
in which there is an orthonormal basis $u_1, \ldots, u_N$
such that $|u_1(x)|^2+\ldots+|u_N(x)|^2\le M^2N$,
for any integer $m\ge C_3 N$ there are points $X_1, \ldots, X_M$ such that
$$
C_2\|f\|_2^2\le \frac{1}{m}\sum\limits_{j=1}^m|f(X_j)|^2\le C_3\|f\|_2^2\quad \forall f\in L.
$$
On the other hand, the probabilistic result of M. Rudelson
from \cite{Rud1}, applied in the case $p=2$
under the same assumption that $|u_1(x)|^2+\ldots+|u_N(x)|^2\le M^2N$
for some
orthonormal basis $u_1, \ldots, u_N$,
provides the discretization result (with high probability) only
for $m=O(N\log N)$ points. Moreover, it is known, that for general
distributions this additional $\log N$ factor cannot be removed
(see also the discussion in \cite{DPTT}, \cite{Tem1}, and \cite{Tem2}).

The assumption that
$|u_1(x)|^2+\ldots+|u_N(x)|^2\le M^2N$ for some constant $M>0$,
for some
orthonormal basis $u_1, \ldots, u_N$
is equivalent to the bound
$$
\|f\|_\infty\le M\sqrt{N}\|f\|_2\quad \forall f\in L
$$
and actually for every orthonormal basis in $L$ the initial bound is
true (see \cite[Proposition 2.1]{5-2}).
We also note that the constant $M$ cannot be less than~$1$,
which will be often used throughout the proofs without mentioning.
Lewis' change of density theorem (see \cite{Lew} or \cite{SZ}) implies that
one can always find a new measure $\nu$
such that the space $(L, \|\cdot\|_{L^p(\mu)})$ is linearly isometric
to some space
$(L', \|\cdot\|_{L^p(\nu)})$ and the space $L'$ already possesses
the desired orthonormal basis with $M=1$.
This observation is very useful when we study discretization with weights.

For a general $p\in[1,\infty)$ one can consider a similar general
assumption on the subspace $L\subset L^p(\mu)$:
for some $q\in[1, \infty)$ and for some constant $M>0$ one has
$$
\|f\|_\infty\le MN^{1/q}\|f\|_q\quad \forall f\in L.
$$
We call this type of assumption the $(\infty, q)$ Nikolskii-type inequality assumption
(with constant~$M$)
after S.M. Nikolskii who proved such inequalities
for multivariate trigonometric polynomials (see \cite{Nik} or \cite[Theorem 3.3.2]{TemBook}).
Our two main results
concerning sampling discretization under the Nikolskii-type inequality assumption
is collected in the following theorem
(see Corollary \ref{cor4} and Corollary~\ref{cor6}).

{\bf Theorem A.}{\it\
Let $p\in(1,\infty)$, $M\ge 1$, $\varepsilon\in(0, 1)$.
There is a positive constant $C:=C(M, p, \varepsilon)$ such that
for every $N$-dimensional subspace $L$ of $L^p(\mu)\cap C(\Omega)$, for which
$$
\|f\|_\infty\le MN^{\frac{1}{\max\{p,2\}}}\|f\|_{\max\{p,2\}}\quad \forall f\in L,
$$
for every integer
$m\ge CN[\log N]^{\max\{p,2\}}$
there are points $X_1,\ldots, X_m$ such that
$$
(1-\varepsilon)\|f\|_p^p\le\frac{1}{m}\sum\limits_{j=1}^m|f(X_j)|^p
\le (1+\varepsilon)\|f\|_p^p\quad \forall f\in L.
$$
}

This theorem improves the recently obtained results
from \cite{5-1} and \cite{5-2}, where
the
sampling discretization
was established for any $m\ge CN[\log N]^3$ points,
for any $p\in[1, 2)$ provided that
the $(\infty, 2)$ Nikolskii-type inequality holds (see \cite[Theorem 2.2]{5-2}).
We point out that our approach does not improve the estimate
for the number of discretizing points in the case $p=1$.
For any $p\in[1, \infty)$ the two cited papers provide
the following general conditional result (see \cite[Theorem 1.3]{5-1}).
Let $p\in[1, \infty)$ and let $L$ be an $N$-dimensional subspace of $L^p(\mu)\cap C(\Omega)$.
Assume that for the entropy numbers (see Definition \ref{D-ent})
of the unit ball $B_p(L):=\{f\in L\colon \|f\|_p\le 1\}$
with respect to the uniform norm $\|\cdot\|_\infty$
one has
\begin{equation}\label{EntNumbCond}
e_k(B_p(L), \|\cdot\|_\infty)\le MN^{1/p}2^{-k/p}\quad 0\le k\le \log N.
\end{equation}
Then for any integer $m\ge C(M, p, \varepsilon)N[\log N]^2$ there are points $X_1, \ldots, X_m$
such that
$$
(1-\varepsilon)\|f\|_p^p\le\frac{1}{m}\sum\limits_{j=1}^m|f(X_j)|^p
\le (1+\varepsilon)\|f\|_p^p\quad \forall f\in L.
$$
Paper \cite{5-2} then provides good bounds for the mentioned entropy numbers, but
only for $p\in[1, 2]$. Instead,
our approach uses bounds for the entropy numbers with respect to the
discretized uniform (semi)norm $\|f\|_{\infty, X}:=\max\limits_{1\le j\le m}|f(X_j)|$
for a discrete set of $m$ points $X:=\{X_1, \ldots, X_m\}$.
These bounds for $p\ge2$ are
known (see \cite[Lemma~16.5.4]{Tal} and \cite{Tal-Sect})
and for $p\in(1, 2)$ we deduce them in Appendix C
(the proof is similar to the proof of \cite[Proposition 16.8.6]{Tal}).
This allows to obtain the new result for $p>2$ under the
$(\infty, p)$ Nikolskii-type inequality assumption
and, for $p\in(1, 2)$, to improve the bound for the number of discretizing points
from \cite[Theorem 2.2]{5-2}.
We note that the $(\infty, p)$ Nikolskii-type inequality,
which is assumed in Theorem A for $p>2$,
provides an
estimate for the diameter of the unit ball $B_p(L)$
with respect to the uniform norm $\|\cdot\|_\infty$.
Thus, in place of the assumptions on all the entropy numbers,
Nikolskii-type inequality assumption restricts the behaviour
of only the first entropy number
$e_0(B_p(L), \|\cdot\|_\infty)$.
Therefore, our bound for $p>2$ is obtained
under the less restrictive assumptions
but provides a little worse dependence on the dimension
compared to the bound from \cite{5-1}
under the assumptions \eqref{EntNumbCond}.

The approach that we use is based on Talagrand's generic chaining technique (see \cite{Tal})
and combines the ideas from \cite{GueRud} on the symmetrization argument,
the new developments in chaining technique from \cite{VH}, and
some known bounds for the entropy numbers from \cite{Tal-Sect} and \cite{Tal-Emb},
which can also be found in the book~\cite{Tal}.
It should be mentioned that the chaining technique
has already been used in various works on sampling discretization
(see \cite{Tem1}, \cite{Tem2}, and \cite{Tem3}),
on learning theory (see \cite{KoTe} and \cite[Chapter 4]{TemBook2}),
and on the problem of approximation of one-dimensional marginals
(see \cite{Rud1}, \cite{Rud2}, \cite{GueRud}, and \cite{GMPT-J}),
and proved to be a powerful tool in these areas.

As it has already been mentioned above, the main results of the present paper are deduced from
several general estimates of the expectation $\mathbb{E}\bigl[V_p(B)\bigr]$
for a $\theta$-convex symmetric set $B\subset L$ (see Definition \ref{D-q-conv}).
The main technical result of our work is Theorem \ref{T0},
where bounds for the expectation $\mathbb{E}\bigl[V_p(B)\bigr]$
are obtained under a certain decay rate assumption on the entropy numbers
$e_k(B, \|\cdot\|_{\infty, X})$.
Then, using bounds for this entropy numbers (see Corollary \ref{ent-est} and Lemma \ref{ent-est-2}),
we obtain estimates on the
expectation $\mathbb{E}\bigl[V_p(B)\bigr]$ for general $\theta$-convex sets in $L$
and for the $L^p(\mu)$ unit balls $B_p(L)$.
In particular, we show (see Corollary \ref{cor3})
that for any symmetric $\theta$-convex body $B\subset L$
and for any
$p\in[\theta,\infty)$ one has
\begin{equation}\label{conv-bound}
\mathbb{E}\sup\limits_{f\in B}\Bigl|\frac{1}{m}\sum\limits_{j=1}^m|f(X_j)|^p - \|f\|_p^p\Bigr|
\le C\bigl(A +
A^{\frac{1}{\theta}}
(\sup\limits_{f\in B}\mathbb{E}|f(X_1)|^p)^{1-\frac{1}{\theta}}\bigr),
\end{equation}
where
$$
A= \frac{[\log m]^{\theta}}{m}
\mathbb{E}\bigl(\sup\limits_{f\in B}\max\limits_{1\le j\le m}|f(X_j)|^{p}\bigr).
$$
Since the ball $B_p(L)$ is $p$-convex when $p\ge 2$, this estimate
implies Theorem A for $p>2$.
The obtained bound is closely related to the theorem of O. Gu$\acute{e}$don and M. Rudelson
from \cite{GueRud}
which asserts (we formulate the result in our terms of functional spaces)
that for any
$\theta$-convex body $B\subset L$ contained in some Euclidean ball $D$
for any
$p\in[\theta,\infty)$ one has
\begin{equation}\label{GR-bound}
\mathbb{E}\bigl[V_p(B)\bigr]
\le C\bigl(A +
A^{1/2}
(\sup\limits_{f\in B}\mathbb{E}|f(X_1)|^p)^{1/2}\bigr),
\end{equation}
where
$$
A= \frac{[\log m]^{2(1-\frac{1}{\theta})}}{m}
\mathbb{E}\bigl(\sup\limits_{f\in D}\max\limits_{1\le j\le m}|f(X_j)|^2
\sup\limits_{h\in B}\max\limits_{1\le j\le m}|h(X_j)|^{p-2}\bigr).
$$
The approach of our paper based on R. Van Handel's Theorem \ref{T-VH}
allows us to improve the power
of logarithm in this result.
We prove (see Corollary \ref{cor1}) that in the same setting as above one actually has
the bound \eqref{GR-bound} with
$$
A=
\frac{1}{m}\mathbb{E}\bigl(\sup\limits_{f\in D}\max\limits_{1\le j\le m}|f(X_j)|^2
\sup\limits_{h\in B}\max\limits_{1\le j\le m}|h(X_j)|^{p-2}\bigr)
+
\frac{\log m}{m}\mathbb{E}\bigl(\sup\limits_{h\in B}
\max\limits_{1\le j\le m}|h(X_j)|^{p}\bigr).
$$
We note that the convexity parameter $\theta$ cannot be less than $2$
implying that $1\le 2(1-\frac{1}{\theta})$.
The assumption that $B$ is contained in some Euclidean ball
allows to use better bounds for the entropy numbers,
which reduces the power of logarithm compared to the estimate
\eqref{conv-bound}. The drawback is that we have to use
the quantity
$\sup\limits_{f\in D}\max\limits_{1\le j\le m}|f(X_j)|$ which in general is larger
than
$\sup\limits_{h\in B}\max\limits_{1\le j\le m}|h(X_j)|$.
When we consider $B=B_p(L)$ with $p\ge2$, we can take $D=B_2(L)$
and then $B\subset D$. Nevertheless, under the $(\infty, p)$
Nikolskii-type inequality assumption with constant $M$, we can only guarantee
the bound $\|f\|_\infty\le M^{p/2}\sqrt{N}\|f\|_2$
which implies that $A\le \frac{1}{m}M^{2p-2}N^{2-\frac{2}{p}} + \frac{\log m}{m}N$.
This means that even the application of our sharper version of
Gu$\acute{e}$don--Rudelson
bound
still implies only polynomial dependence of the number of discretizing points
on the dimension
for the
initial problem of sampling discretization
(under the $(\infty, p)$ Nikolskii-type inequality assumption).
Thus, we inclined to use
the estimate \eqref{conv-bound} to obtain almost linear dependence
from Theorem A.

We also mention that the obtained sharper version of the Gu$\acute{e}$don--Rudelson
bound \eqref{GR-bound} implies (see Corollary \ref{cor2}) that under the
$(\infty, 2)$ Nikolskii-type inequality
assumption with constant $1$ for any $p\ge2$ one has
$$
\mathbb{E}\bigl[V_p(B_p(L))\bigr]
\le C\Bigl(\frac{\log m}{m} N^{p/2} + \Bigl[\frac{\log m}{m} N^{p/2}\Bigr]^{1/2}\Bigr).
$$
Thus,
for any integer $m\ge c(\varepsilon, p) N^{p/2}\log N$
there are points $X_1, \ldots, X_m$
such that
$$
(1-\varepsilon)\|f\|_p^p\le \sum_{j=1}^{m}|f(X_j)|^p\le (1+\varepsilon)\|f\|_p^p\quad \forall f\in L.
$$
The combination (see Remark \ref{rem-weights}) of this observation with
Lewis' change of density theorem (see \cite{Lew} or \cite{SZ})
implies that for any $p\ge2$ and already for any $N$-dimensional subspace $L\subset L^p(\mu)$,
for any integer $m\ge c(\varepsilon, p) N^{p/2}\log N$
there are points $X_1, \ldots, X_m$
and positive numbers (weights) $\lambda_1,\ldots, \lambda_m$ such that
$$
(1-\varepsilon)\|f\|_p^p\le \sum_{j=1}^{m}\lambda_j|f(X_j)|^p\le
(1+\varepsilon)\|f\|_p^p\quad \forall f\in L.
$$
This gives a slightly different proof
for the theorem of  J. Bourgain, J. Lindenstrauss, and V. Milman
concerning good embeddings
of finite dimensional subspaces of $L^p[0, 1]$ into $\ell_p^m$ (see \cite[Theorem 7.3]{BLM}).
Their theorem asserts that for any $N$-dimensional subspace $L$ of $L^p[0, 1]$
there is an $N$-dimensional subspace $L'$ in $\ell_p^m$, with $m=c(\varepsilon, p)N^{p/2}\log N$,
at a Banach-Mazur distance not greater than $1+\varepsilon$ from $L$.
We note that the approach in \cite{BLM} is also probabilistic and
also uses empirical distributions.
The mentioned embedding problem
is closely related to our initial question concerning sampling discretization.
We note that in the case $p\in(1, 2)$ M. Talagrand managed to prove
(see \cite{Tal-Emb} or \cite[Theorem 16.8.1]{Tal}) that for an $N$-dimensional subspace $L$ of $L^p[0, 1]$
there is an $N$-dimensional subspace $L'$ in $\ell_p^m$, with $m=c(\varepsilon, p)N\log N[\log\log N]^2$,
at a Banach-Mazur distance not greater than $1+\varepsilon$ from $L$.
Our results imply (see Remark~\ref{rem-weights-2})
that for any number $p\in(1, 2)$ and for any $N$-dimensional subspace $L\subset L^p(\mu)$
for any integer $m\ge c(\varepsilon, p) N[\log N]^2$
there are points $X_1, \ldots, X_m$
and positive numbers (weights) $\lambda_1,\ldots, \lambda_m$ such that
$$
(1-\varepsilon)\|f\|_p^p\le \sum_{j=1}^{m}\lambda_j|f(X_j)|^p\le
(1+\varepsilon)\|f\|_p^p\quad \forall f\in L.
$$
Thus, it will be interesting to understand if it is possible
to reach (or even improve) Talagrand's bound for the dimension $m$
in the embedding problem
by means of the sampling discretization with weights.
More information concerning the embedding
problem can be found in the
expository paper by W.B.~Johnson and G.~Schechtman \cite{JS}.

We also obtain the analog of the Gu$\acute{e}$don--Rudelson
bound \eqref{GR-bound} when one assumes the inclusion of the $\theta$-convex set $B$
not in an Euclidean ball but in another $q$-convex body:
if $B\subset D\subset L$, where $B$ is $\theta$-convex
and $D$ is $q$-convex, then
for any $p\in[\max\{\theta, q\},\infty)$
one has
$$
\mathbb{E}\sup\limits_{f\in B}\Bigl|\frac{1}{m}\sum\limits_{j=1}^m|f(X_j)|^p - \|f\|_p^p\Bigr|
\le C\bigl(A + A^{\frac{1}{q}}
(\sup\limits_{f\in B}\mathbb{E}|f(X_1)|^p)^{1-\frac{1}{q}}\bigr),
$$
where
$$
A=
\frac{[\log m]^q}{m}\mathbb{E}\bigl(\sup\limits_{f\in D}\max\limits_{1\le j\le m}|f(X_j)|^q
\sup\limits_{h\in B}\max\limits_{1\le j\le m}|h(X_j)|^{p-q}\bigr)
+
\frac{\log m}{m}\mathbb{E}\bigl(\sup\limits_{h\in B}
\max\limits_{1\le j\le m}|h(X_j)|^{p}\bigr).
$$

Further the paper is organized as follows.
In the second section we recall the basic notions of the chaining technique,
formulate some extensions of the results from \cite{VH},
and give some technical
lemmas that are used further. In the third section
we obtain
bounds for the expectation of the random variable $V_p(B)$
for $\theta$-convex sets $B$
under the assumptions on the decay rate of the
entropy numbers of the set $B$
with respect to the discretized uniform norm
$\|f\|_{\infty, X}:=\max\limits_{1\le j\le1}|f(X_j)|$
for a fixed set of points $X:=\{X_1,\ldots, X_m\}$.
Finally, in the fourth section we prove the main results of
the paper concerning the sampling discretization in subspaces of $L^p(\mu)$
along with some general bounds for the expectation of $V_p(B)$
for $\theta$-convex sets $B$.
Appendices A and B contain the proofs of the extensions of the results from \cite{VH},
which we are using in the paper.
However, we note that they repeat the proofs from \cite{VH} almost word for word
and are presented here only for the readers' convenience.
In Appendix C we provide the bound for the entropy numbers
of the ball $B_p(L)$, $p\in(1, 2)$, with respect to the norm $\|\cdot\|_{\infty, X}$.

Throughout the paper the symbols $c, c_1, c_2, C, C_1, C_2, \ldots$
denote absolute constants whose values may vary from line to line.
Similarly, the symbols $c(a, b, \ldots), c_1(a, b, \ldots), c_2(a, b, \ldots),
C(a,b,\ldots)$, $C_1(a,b,\ldots)$, $C_2(a,b,\ldots), \ldots$
denote numbers whose values depend only on parameters $a,b,\ldots$,
and also may vary from line to line.
If the random variable $X$ has the distribution $\mu$, we write $\mathbb{E}_Xf(X)$ (or simply $\mathbb{E}f(X)$)
in place of the integral
$\displaystyle \int_\Omega f\, d\mu$.

\section{Generic chaining, van Handel's approach and auxiliary lemmas}

We recall the basic facts from the generic chaining theory (see \cite{Tal}).

Let $\varepsilon_f$ be a random process with $f\in(F, \varrho)$
where $\varrho$ is a quasi-metric on $F$, i.e.
it has all the properties of a metric but, in place of usual
triangle inequality, one has the following relaxed triangle inequality
\begin{equation}\label{triangle}
\varrho(f,g)\le R(\varrho(f,h)+\varrho(h,g))
\end{equation}
for some constant $R>0$ for all $f,g,h\in F$.
Assume that there are numbers $K>0$ and $\alpha>0$ such that
\begin{equation}\label{cond}
P(|\varepsilon_f-\varepsilon_g|\ge Kt^{1/\alpha}\varrho(f,g))\le  2e^{-t}
\end{equation}
for all $t>0$.

\begin{definition}
An admissible sequence of $F$ is an increasing sequence $(\mathcal{F}_k)$ of partitions of $F$
such that $|\mathcal{F}_k|\le 2^{2^k}$ for all $k\ge1$ and $|\mathcal{F}_0|=1$.
For $f\in F$ let $F_k(f)$ denote the unique element of $\mathcal{F}_k$ that contains $f$.
\end{definition}

\begin{definition}
Let $\alpha>0$ and $\theta\ge1$. Let
$$
\gamma_{\alpha,\theta}(F,\varrho):=
\Bigl(\inf\sup_{f\in F}\sum\limits_{k=0}^\infty
\bigl[2^{k/\alpha} {\rm diam}\bigl(F_k(f)\bigr)\bigr]^\theta\Bigr)^{1/\theta},
$$
where ${\rm diam}(G):=\sup\limits_{f,g\in G}\varrho(f,g)$
and where the infimum is taken over all admissible sequences of $F$ .
\end{definition}
The quantity $\gamma_{\alpha,\theta}(F,\varrho)$ is called the chaining functional.
If the metric $\varrho$ is induced by a norm $\|\cdot\|$, we will
also use the notation $\gamma_{\alpha,\theta}(F,\|\cdot\|)$ in place of $\gamma_{\alpha,\theta}(F,\varrho)$.

We need the following fundamental result (see \cite[Theorem 2.2.22]{Tal}).

\begin{theorem}\label{T-Tal}
Under the above assumptions \eqref{triangle} and \eqref{cond}
there is a number $C:=C(\alpha, K, R)$,
dependent only on the
parameters $\alpha, K, R$, such that
for any $f_0\in F$ one has
$$
\mathbb{E}\sup\limits_{f\in F}|\varepsilon_f-\varepsilon_{f_0}|
\le C\gamma_{\alpha,1}(F,\varrho).
$$
\end{theorem}

We note that in \cite{Tal} the theorem is stated only for a metric $\varrho$
and in the case when $\alpha=2$, but Theorem \ref{T-Tal} can be proved essentially
repeating the argument from \cite{Tal}.

\begin{definition}\label{D-ent}
Recall the definition of the entropy numbers:
$$
e_k(F,\varrho):=
\inf\Bigl\{\varepsilon\colon \exists f_1,\ldots, f_{n_k}\in F\colon
F\subset \bigcup\limits_{j=1}^{n_k}B_\varepsilon(f_j)\Bigr\},
$$
where $n_k=2^{2^k}$ for $k\ge 1$ and $n_0=1$
and where
$B_\varepsilon(f):=\{g\colon \varrho(f,g)<\varepsilon\}$.
\end{definition}
If the metric $\varrho$ is induced by a norm $\|\cdot\|$, we will
also use the notation $e_k(F,\|\cdot\|)$ in place of $e_k(F,\varrho)$.
We note here that sometimes the other definition of the entropy numbers is used with
$2^k$ points in place of $2^{2^k}$.

We will also use the following property of the entropy numbers in an $N$-dimensional space
(see estimate (7.1.6) in \cite{TemBook} and Corollary 7.2.2 there).
Assume that $\varrho$ is induced by a norm $\|\cdot\|$. Then for $k>k_0$ one has
\begin{equation}\label{Eq1}
e_k(F,\|\cdot\|)\le 3\ 2^{2^{k_0}/N} e_{k_0}(F,\|\cdot\|) 2^{-2^k/N}.
\end{equation}

\begin{definition}\label{D-q-conv}
Let $L$ be a linear space endowed with a norm $\|\cdot\|$.
This norm is called
$q$-convex (with constant $\eta>0$) if
$$
\Bigl\|\frac{f+g}{2}\Bigr\|\le \max(\|f\|,\|g\|)-\eta\|f-g\|^q
$$
for any $f, g$ with $\|f\|\le1, \|g\|\le1$.

A symmetric convex body $D\subset L$ is called
$q$-convex (with constant $\eta>0$) if it is the unit ball
of some $q$-convex (with constant $\eta>0$) norm $\|\cdot\|$
on $L$, i.e. $D=\{f\in L\colon \|f\|\le 1\}$.
\end{definition}

We will use the following fundamental result from \cite{VH}.

\begin{theorem}\label{T-VH}
Let $q\ge2$, $p>1$, $\alpha>0$.
Let $L$ be a linear space and let $D\subset L$
be a symmetric $q$-convex (with constant $\eta>0$)
body.
Let $\varrho$ be a quasi-metric on $L$
such that
$$
\varrho(f,g)\le R\bigl(\varrho(f,h)+\varrho(h,g)\bigr);\quad \varrho\Bigl(f, \frac{f+g}{2}\Bigr)\le \varkappa \varrho(f,g)
$$
for all $f,g,h\in L$, for some constants $R,\varkappa>0$.
Assume that there is a metric $d$ on $L$
and for each $h\in L$ there is a norm $\|\cdot\|_h$ on $L$ such that
$$
c_1 d(f,g)^p\le \varrho(f,g)\le c_2\bigl(\|f-g\|_h+d(f,g)(d(f,h)^{p-1}+d(h,g)^{p-1})\bigr)
$$
for some numbers $c_1, c_2>0$.
Then there is a number $C:=C(q,p,\alpha, R,\varkappa, c_1, c_2)$
such that for any $B\subset D$ one has
$$
\gamma_{\alpha,1}(B, \varrho)\le
C\Bigl(
\eta^{-1/q}\Bigl[\sup\limits_{h\in B}
\sum\limits_{k=0}^\infty\bigl(2^{k/\alpha}e_k(D, \|\cdot\|_h)\bigr)^{\frac{q}{q-1}}\Bigr]^{\frac{q-1}{q}}
+
\bigl[\gamma_{\alpha p,p}(B, d)\bigr]^p
\Bigr).
$$
\end{theorem}

The quasi-metric $\varrho$ in the above theorem can appear from the expressions of the following type
$$
\widetilde{\varrho}(f,g):=
\Bigl(\int\bigl||f|^p-|g|^p\bigr|^r\, d\nu\Bigr)^{1/r}=\||f|^p-|g|^p\|_{L^r(\nu)}
$$
for some positive (not necessarily probability) measure $\nu$, $p>1, r\in[1,\infty]$.
Indeed, set
\begin{multline*}
\varrho(f,g):=
  \||f-g|(|f|^{p-1}+|g|^{p-1})\|_{L^r(\nu)};
  \\
  \|f\|_h:=\||f||h|^{p-1}\|_{L^r(\nu)};
  \\
  d(f,g):=\||f-g|^p\|_{L^r(\nu)}^{1/p}=\|f-g\|_{L^{pr}(\nu)}.
\end{multline*}
It can be readily verified that $\widetilde{\varrho}(f,g)\le p\varrho(f,g)$.

\begin{lemma}\label{Lem}
For the quasi metric $\varrho$, metric $d$ and norms $\|\cdot\|_h$ defined above we have
$$
\varrho(f,g)\le C_1(p)\bigl(\varrho(f,h)+\varrho(h,g)\bigr);\quad
\varrho\Bigl(f, \frac{f+g}{2}\Bigr)\le \varrho(f,g);
$$
$$
C_2(p)d(f,g)^p\le \varrho(f,g)\le C_3(p)\bigl(\|f-g\|_h+d(f,g)(d(f,h)^{p-1}+d(h,g)^{p-1})\bigr)
$$
for some numbers $C_1(p), C_2(p), C_3(p)$, dependent only on $p>1$.
\end{lemma}

\begin{proof}
We note that
$$
(|f|+|g|)^{p-1}\le 2^{p-1}\max\{|f|^{p-1}, |g|^{p-1}\}\le 2^{p-1}(|f|^{p-1}+|g|^{p-1})
$$ for $p>1$.
Thus,
\begin{multline*}
2^{1-p}|f-g|^p\le |f-g|(|f|^{p-1}+|g|^{p-1})=|f-g|(|f-h+h|^{p-1}+|g-h+h|^{p-1}
\\
\le 2^{p-1}|f-g|(|f-h|^{p-1}+|h|^{p-1}+|g-h|^{p-1}+|h|^{p-1})
\\
\le
2^p(|f-g||h|^{p-1} + |f-g||f-h|^{p-1}+|f-g||g-h|^{p-1})
\end{multline*}
implying, by triangle and H$\ddot{o}$lder's inequalities, the estimates
$$
2^{1-p}d(f,g)^p
\le \varrho(f,g)
\le
2^{p}\bigl(\|f-g\|_h + d(f,g)(d(f,h)^{p-1}+d(g,h)^{p-1})\bigr).
$$
Next,
\begin{multline*}
\varrho\Bigl(f, \frac{f+g}{2}\Bigr)
=
2^{-1}\Bigl\||f-g|\Bigl(|f|^{p-1}+\Bigl|\frac{f+g}{2}\Bigr|^{p-1}\Bigr)\Bigr\|_r
\\
\le
2^{-1}\||f-g|(|f|^{p-1}+|f|^{p-1}+|g|^{p-1})\|_r
\le
\varrho(f,g).
\end{multline*}
Finally,
\begin{multline*}
|f-g|(|f|^{p-1}+|g|^{p-1})
\\
\le
2^{p-1}\bigl(|f-h|(|f|^{p-1}+|h|^{p-1}+|h-g|^{p-1})+|h-g|(|f-h|^{p-1}+|h|^{p-1}+|g|^{p-1})\bigr)
\\
=
2^{p-1}\bigl(|f-h|(|f|^{p-1}+|h|^{p-1}) + |h-g|(|h|^{p-1}+|g|^{p-1})
\\
+
|f-h||h-g|^{p-1} + |h-g||f-h|^{p-1}\bigr).
\end{multline*}
We now note that for any positive numbers $a, b$ by
Young's inequality one has $ab^{p-1}\le a^p+b^p$.
Thus,
\begin{multline*}
|f-h||h-g|^{p-1} + |h-g||f-h|^{p-1}\le 2(|f-h|^p+|h-g|^p)
\\
\le 2^p\bigl(|f-h|(|f|^{p-1}+|h|^{p-1}) +
|h-g|(|h|^{p-1}+|g|^{p-1})\bigr)
\end{multline*}
and
$$
\varrho(f,g)
\le 4^p\bigl(\varrho(f,h)+\varrho(h,g)\bigr).
$$
The lemma is proved.
\end{proof}

\begin{remark}{\rm
We note that in \cite{VH} only a special case of Theorem \ref{T-VH}
was considered (see Theorem 7.3 there), but the proof of Theorem \ref{T-VH}
repeats the argument there almost verbatim.
We will provide the details in Appendix A for the readers' convenience.
}
\end{remark}

We need the following bound (see \cite[Theorem 4.1.4]{Tal} and \cite[Theorem~5.8]{VH}).

\begin{theorem}\label{T-Conv}
Let $B$ be a symmetric $q$-convex (with constant $\eta$)
body in some linear space $L$ and
let $\|\cdot\|$ be a norm on $L$.
Then
for any $\alpha>0$ there is a number $C(\alpha, q)>0$
such that
$$
\gamma_{\alpha,q}(B,\|\cdot\|)\le C(\alpha, q) \eta^{-1/q}
\sup\limits_{k\ge0}2^{k/\alpha}e_k(B,\|\cdot\|).
$$
\end{theorem}

We also need the following extension of the above result.

\begin{theorem}\label{T-Conv-2}
Let $B$ be a symmetric $q$-convex (with constant $\eta$)
body in some linear space $L$ and
let $\|\cdot\|$ be a norm on $L$.
Then
for any $\alpha>0$ and for any $p\in[1,q)$
there is a number $C(\alpha,p,q)>0$
such that
$$
\gamma_{\alpha,p}(B, \|\cdot\|)\le
C(\alpha,p,q)\eta^{-p/q}\Bigl(\sum_{k\ge0}(2^{k/\alpha}e_k(B, \|\cdot\|))^{\frac{pq}{q-p}}
\Bigr)^{\frac{q-p}{pq}}.
$$
\end{theorem}

The proof again repeats the argument from \cite[Theorem 5.8]{VH} almost verbatim.
We present the proof in Appendix B for the readers' convenience.

Finally, we will use the following technical bound.
\begin{lemma}\label{L-bound}
Let $a, b>0$. Then there is a number $C(a,b)>0$ such that
$$
\sum\limits_{k\ge \log m}(2^{a k}2^{-2^k/m})^b
\le C(a,b) m^{ab}\quad \forall m\ge2.
$$
\end{lemma}

\begin{proof}
Note that
$$
m^{-ab}\sum\limits_{k\ge \log N}(2^{a k-2^k/m})^b
=
\sum\limits_{k\ge \log m}(2^{a (k-\log m)-2^{k-\log m}})^b.
$$
There is a number $c(a)>0$ such that $a x- 2^x\le -x + c(a)$ for any $x>0$.
Thus, the last expression is estimated by
$$
2^{bc(a)}\sum\limits_{k\ge \log m}(2^{-(k-\log m)})^b\le C(a,b).
$$
The lemma is proved.
\end{proof}

\section{Discretization under the entropy numbers decay rate assumption}

Let $X_1, \ldots, X_m$ be independent identically distributed random variables
and
let $B$ be a set of functions. We consider the following random variables:
$$
V_p(B):=
\sup\limits_{f\in B}\Bigl|\frac{1}{m}\sum\limits_{j=1}^m|f(X_j)|^p - \|f\|_p^p\Bigr|,\quad
R_p(f) = \sum\limits_{j=1}^m|f(X_j)|^p
$$
In this section we provide conditional bounds for the expectation $\mathbb{E}\bigl[V_p(B)\bigr]$
under the assumptions on the decay rate of
the entropy numbers of the set $B$ with respect to some discretized uniform norm.

Following the ideas of O.~Gu$\acute{e}$don and M.~Rudelson from \cite{GueRud}
we start with the following symmetrization argument.

\begin{lemma}\label{Lem-1}
Assume that there is a number $\delta\in(0,1)$ such that,
for every fixed set of $m$ points $X:=\{X_1,\ldots, X_m\}$, for some number $\Theta(X)$,
one has
$$
\mathbb{E}_\varepsilon\sup\limits_{f\in B}\Bigl|\sum\limits_{j=1}^m\varepsilon_j|f(X_j)|^p\Bigr|\le
\Theta(X)\sup\limits_{f\in B}\bigl(R_p(f)\bigr)^{1-\delta}
$$
where $\varepsilon_1,\ldots, \varepsilon_m$ are independent symmetric Bernoulli random variables with values $\pm1$.
Then
$$
\mathbb{E}V_p(B)\le
2^{1/\delta}m^{-1}\mathbb{E}[\Theta(X)^{1/\delta}]
+ 2\delta^{-1} \bigl(m^{-1}\mathbb{E}[\Theta(X)^{1/\delta}]\bigr)^\delta
\Bigl(\sup\limits_{f\in B}\mathbb{E}|f(X_1)|^p\Bigr)^{1-\delta}.
$$
\end{lemma}

\begin{proof}
Let $X_1',\ldots,X_m'$ be independent copies of $X_1,\ldots,X_m$.
We note that
\begin{multline*}
m\mathbb{E}V_p(B)
=
\mathbb{E}\sup\limits_{f\in B}\Bigl|\sum\limits_{j=1}^m(|f(X_j)|^p - \mathbb{E}|f(X_j')|^p)\Bigr|
\le
\mathbb{E}_X\mathbb{E}_{X'} \sup\limits_{f\in B}\Bigl|\sum\limits_{j=1}^m(|f(X_j)|^p - |f(X_j')|^p)\Bigr|
\\
=
\mathbb{E}_X\mathbb{E}_{X'}\mathbb{E}_\varepsilon
\sup\limits_{f\in B}\Bigl|\sum\limits_{j=1}^m\varepsilon_j(|f(X_j)|^p - |f(X_j')|^p)\Bigr|
\le 2\mathbb{E}_X\mathbb{E}_\varepsilon
\sup\limits_{f\in B}\Bigl|\sum\limits_{j=1}^m\varepsilon_j|f(X_j)|^p\Bigr|
\\
\le
2\mathbb{E}\bigl[\Theta(X)[\sup\limits_{f\in B}R_p(f)]^{1-\delta}\bigr]
\le
2\bigl(\mathbb{E}[\Theta(X)^{1/\delta}]\bigr)^\delta
m^{1-\delta}\Bigl(\mathbb{E}\sup\limits_{f\in B}\frac{1}{m}\sum\limits_{j=1}^m|f(X_j)|^p\Bigr)^{1-\delta}
\\
\le
2\bigl(\mathbb{E}[\Theta(X)^{1/\delta}]\bigr)^\delta  m^{1-\delta}
\Bigl(\mathbb{E}V_p(B) + \sup\limits_{f\in B}\mathbb{E}|f(X_1)|^p\Bigr)^{1-\delta}.
\end{multline*}
Thus,
$$
\mathbb{E}V_p(B)\le 2\bigl(\mathbb{E}[\Theta(X)^{1/\delta}]\bigr)^\delta  m^{-\delta}
\Bigl(\mathbb{E}V_p(B) + \sup\limits_{f\in B}\mathbb{E}|f(X_1)|^p\Bigr)^{1-\delta}
$$
and
$$
\mathbb{E}V_p(B)\le 2^{1/\delta}m^{-1}\mathbb{E}[\Theta(X)^{1/\delta}]
+ 2\delta^{-1} \bigl(m^{-1}\mathbb{E}[\Theta(X)^{1/\delta}]\bigr)^\delta
\Bigl(\sup\limits_{f\in B}\mathbb{E}|f(X_1)|^p\Bigr)^{1-\delta}.
$$
Indeed, if for some $v, a, b>0$ and some $\delta\in(0,1)$
one has the estimate $v\le a (v+b)^{1-\delta}$,
then by convexity and Young's inequality one has
$$a(v+b)^{1-\delta}\le av^{1-\delta} + ab^{1-\delta}\le
\delta a^{1/\delta} + (1-\delta)v+ ab^{1-\delta}$$ and
$v\le a^{1/\delta} + \delta^{-1}ab^{1-\delta}$.
The lemma is proved.
\end{proof}

Lemma \ref{Lem-1} reduces the main problem of estimating the expectation
$\mathbb{E}\bigl[V_p(B)\bigr]$
to the estimation of
$$
\mathbb{E}_\varepsilon\sup\limits_{f\in B}\Bigl|\sum\limits_{j=1}^m\varepsilon_j|f(X_j)|^p\Bigr|
$$
for any fixed discrete point set $X=\{X_1,\ldots, X_m\}$.
Thus, we now deal with the Bernoulli random
process $\varepsilon_f:=\sum\limits_{j=1}^m\varepsilon_j|f(X_j)|^p$
and we want to estimate the expectation of its supremum.
For the Bernoulli random
process one has the following tail estimate (see \cite[Lemma 4.3]{LedTal}).

\begin{lemma}\label{TailsEst}
Let $\varepsilon_1,\ldots, \varepsilon_m$ be independent symmetric Bernoulli random variables with values $\pm1$.
Then for any $\tau\in[2, \infty)$ there is a number $C_\tau$, depending only on $\tau$, such that
$$
P\Bigl(\bigl|\sum\limits_{j=1}^m\varepsilon_jc_j\bigr|\ge
C_\tau\bigl(\sum\limits_{j=1}^m|c_j|^{\tau'}\bigr)^{1/\tau'}t^{1/\tau}\Bigr)\le 2e^{-t},
$$
where $\tau'=\frac{\tau}{\tau-1}$.
\end{lemma}

For a fixed discrete set $X=\{X_1,\ldots, X_m\}$
and for any non-negative function $\varphi$ on $X$
we consider the norms
$\|f\|_{r, X; \varphi}:=\Bigl(\sum\limits_{j=1}^m|f(X_j)|^r\varphi(X_j)\Bigr)^{1/r},
r\in[1, \infty)$, defined on all functions $f\colon X\to \mathbb{R}$.
When $\varphi\equiv1$, we write $\|\cdot\|_{r, X}$ in place
of $\|\cdot\|_{r, X; 1}$.
We also set
$\|f\|_{\infty, X}:=\max\limits_{1\le j\le m}|f(X_j)|$.

\begin{lemma}\label{L0}
Let $p\in[1,\infty)$, $q\in[2,\infty)$, $r\in (1,2]$.
Let $X=\{X_1,\ldots, X_m\}$ be a fixed set,
let $L$ be a linear space of functions defined on $X$,
and assume that $D\subset L$ is a symmetric
$q$-convex (with constant $\eta>0$) body.
Then there is a constant $C:=C(p, q, r, \eta)$, which depends only on parameters
$p$, $q$, $r$, and $\eta$,
such that for any $B\subset D$ one has
\begin{multline*}
\mathbb{E}_\varepsilon\sup\limits_{f\in B}\Bigl|\sum\limits_{j=1}^m\varepsilon_j|f(X_j)|^p\Bigr|
\le
C\Bigl(
\Bigl[\sup\limits_{h\in B}
\sum\limits_{k=0}^\infty\bigl(2^{k/r'}e_k(D, \|\cdot\|_{r, X; |h|^{r(p-1)}})\bigr)^{\frac{q}{q-1}}\Bigr]^{\frac{q-1}{q}}
+
\bigl[\gamma_{ pr',p}(B, \|\cdot\|_{pr, X})\bigr]^p
\Bigr),
\end{multline*}
where $r'=\frac{r}{r-1}$.
\end{lemma}

\begin{proof} For any $\tau\in[2, \infty)$, by Lemma \ref{TailsEst}, we have the estimate \eqref{cond}
with the quasi-metric
\begin{equation}\label{eq-a-metric}
\varrho_\tau(f,g):=
\Bigl(\sum\limits_{j=1}^m
\bigl||f(X_j)-g(X_j)|(|f(X_j)|^{p-1}+ |g(X_j)|^{p-1})\bigr|^{\tau'}
\Bigr)^{1/\tau'}.
\end{equation}
We chose $\tau = \frac{r}{r-1}=r'$.
Thus, by Theorem \ref{T-Tal}, the bound for the expectation of the supremum over $B$
of the process $\varepsilon_f:=\sum\limits_{j=1}^m\varepsilon_j|f(X_j)|^p$
will follow from the bound
for the chaining functional $\gamma_{\tau,1}(B, \varrho_\tau)$.
By Lemma~\ref{Lem}, we can apply Theorem~\ref{T-VH}
with
$$
\|f\|_h=\|f\|_{r, X; |h|^{r(p-1)}}=\Bigl(\sum\limits_{j=1}^m|f(X_j)|^{r}
|h(X_j)|^{r(p-1)}\Bigr)^{1/r}
$$
and
$$
d(f,g) =
\Bigl(\sum\limits_{j=1}^m|f(X_j) - g(X_j)|^{pr}\Bigr)^{\frac{1}{pr}}
=
\|f-g\|_{pr, X}.
$$
By Theorem~\ref{T-VH}, there is a constant $C:= C(p, q, r, \eta)$ such that
$$
\gamma_{\tau, 1}(B, \varrho_\tau)\le
C\Bigl(
\Bigl[\sup\limits_{h\in B}
\sum\limits_{k=0}^\infty\bigl(2^{k/\tau}e_k(B, \|\cdot\|_h)\bigr)^{\frac{q}{q-1}}\Bigr]^{\frac{q-1}{q}}
+
\bigl[\gamma_{\tau p,p}(B, d)\bigr]^p
\Bigr)
$$
which is the announced bound.
\end{proof}

We now bound the summands of the right hand side of the estimate from the previous lemma
under different assumptions on the bodies $D$ and $B$.

\begin{lemma}\label{L1}
Let $L$ be a linear space of functions defined on a discrete set $X=\{X_1, \ldots, X_m\}$
and let $r\in(1, 2]$, $r':=\frac{r}{r-1}$, $q\ge2$.

$1)$ If $D\subset L$ is a Euclidean unit ball, then
there is a numerical constant $C$ such that
for any $p\in(1,\infty)$ and any $h\in L$ one has
$$
\Bigl[\sum\limits_{k=0}^\infty\bigl(2^{k/2}
e_k(D, \|\cdot\|_{2, X;|h|^{2(p-1)}})\bigr)^2\Bigr]^{1/2}\!\le
\bigl[\sup\limits_{f\in D}\|f\|_{\infty, X}\bigr]
\Bigl(\sum\limits_{j=1}^{m}|h(X_j)|^{2(p-1)}\Bigr)^{1/2}\!.
$$

$2)$ If $D\subset L$, then
for any $p\in(1,\infty)$ and for any $t\in(0, r]$
there is a number $c:=c(p, r, t)$ such that for any $h\in L$ one has
\begin{multline*}
\Bigl[\sum\limits_{k=0}^\infty\bigl(2^{k/r'}
e_k(D, \|\cdot\|_{r, X;|h|^{r(p-1)}})\bigr)^{\frac{q}{q-1}}\Bigr]^{\frac{q-1}{q}}
\\
\le
c
\sup\limits_{f\in D}
\Bigl(\sum\limits_{j=1}^m|f(X_j)|^{pr-t}\Bigr)^{\frac{r-t}{r(pr-t)}}
\Bigl(\sum\limits_{j=1}^m|h(X_j)|^{pr-t}\Bigr)^{\frac{pr-r}{r(pr-t)}}
\Bigl[\sum\limits_{k=0}^\infty\bigl(2^{k/r'}
[e_k(D, \|\cdot\|_{\infty, X})]^{t/r}\bigr)^{\frac{q}{q-1}}\Bigr]^{\frac{q-1}{q}}.
\end{multline*}

$3)$ If $B\subset L$ is $\theta$-convex (with a constant $\zeta>0$) body, then
for any $p\in(1,\infty)$ and for any $s\in(0, pr]$ there is a number $C:=C(p, s, \theta, \zeta)$ such that
$$
\bigl[\gamma_{pr',p}(B, \|\cdot\|_{pr, X})\bigr]^p
\le
C\sup\limits_{f\in B}\Bigl(\sum_{j=1}^{m}|f(X_j)|^{pr-s}\Bigr)^{1/r}
\sup\limits_{k\ge 0}2^{k/r'}[e_k(B,\|\cdot\|_{\infty, X})]^{s/r}
$$
if $p\ge\theta$ and
$$
\bigl[\gamma_{pr',p}(B, \|\cdot\|_{pr, X})\bigr]^p
\le
C\sup\limits_{f\in B}\Bigl(\sum_{j=1}^{m}|f(X_j)|^{pr-s}\Bigr)^{1/r}
\Bigl(\sum_{k\ge0}\bigl(2^{k/r'}[e_k(B,\|\cdot\|_{\infty,X})]^{s/r}\bigr)^{\frac{\theta}{\theta-p}}
\Bigr)^{\frac{\theta-p}{\theta}}
$$
if $p\in(1, \theta)$.
\end{lemma}

\begin{proof}

$1)$ The first claim has been observed in \cite{VH} (see the proof of Corollary~7.4 there)
and follows from the bounds for the entropy numbers
of ellipsoids with respect to a Euclidean norm
from \cite[Lemma 2.5.5]{Tal}. The cited lemma implies that
$e_{k+3}(D, \|\cdot\|)\le 3\max\limits_{i\le k}(a_{2^i}2^{i-k})$
for any Euclidean ball $D$
and for any norm $\|\sum c_iu_i\|=(\sum a_i^2 c_i^2)^{1/2}$, where $\{u_i\}$ is an orthonormal
basis in $L$ with respect to the norm generated by the Euclidean ball $D$ and where
$\{a_i\}$ is a non-increasing sequence of positive numbers.
Thus,
\begin{multline*}
\Bigl[\sum\limits_{k=3}^\infty\bigl(2^{k/2}
e_k(D, \|\cdot\|)\bigr)^2\Bigr]^{1/2}\le
c_1\Bigl[\sum\limits_{k=0}^\infty2^{k}
\sum\limits_{i\le k}\bigl(a_{2^i}2^{i-k}\bigr)^2\Bigr]^{1/2}
=
c_1\Bigl[\sum\limits_{i}a_{2^i}^22^{2i}
\sum\limits_{k=i}^\infty2^{-k}\Bigr]^{1/2}
\\=
c_2\Bigl[\sum\limits_{i}a_{2^i}^22^{i}\Bigr]^{1/2}
\le
c_2\Bigl[\sum\limits_{i}a_{i}^2\Bigr]^{1/2}
= c_2\Bigl[\sum\limits_{i}\|u_i\|^2\Bigr]^{1/2}.
\end{multline*}
In our case
\begin{multline*}
\Bigl[\sum\limits_{k=0}^\infty\bigl(2^{k/2}
e_k(D, \|\cdot\|_{2, X;|h|^{2(p-1)}})\bigr)^2\Bigr]^{1/2}
\le
c_2\Bigl[\sum\limits_{i}\|u_i\|_{2, X;|h|^{2(p-1)}}^2\Bigr]^{1/2}
\\=
c_2\Bigl[\sum\limits_{i}\sum\limits_{j=1}^{m}|u_i(X_j)|^2|h(X_j)|^{2(p-1)}\Bigr]^{1/2}
\le
c_2\max\limits_{1\le j\le m}\Bigl[\sum\limits_{i}|u_i(X_j)|^2\Bigr]^{1/2}
\Bigl[\sum\limits_{j=1}^{m}|h(X_j)|^{2(p-1)}\Bigr]^{1/2}
\\
=
c_2\max\limits_{1\le j\le m}\sup\limits_{\sum c_i^2\le 1}\Bigl|\sum\limits_{i}c_iu_i(X_j)\Bigr|
\Bigl[\sum\limits_{j=1}^{m}|h(X_j)|^{2(p-1)}\Bigr]^{1/2}
=c_2\bigl[\sup\limits_{f\in D}\|f\|_{\infty, X}\bigr]
\Bigl(\sum\limits_{j=1}^{m}|h(X_j)|^{2(p-1)}\Bigr)^{1/2}\!.
\end{multline*}
The first claim is proved.

$2)$ For $p>1$ and $t\in(0, r]$ one has
\begin{multline*}
\|f\|_{r, X; |h|^{r(p-1)}}=\Bigl(\sum\limits_{j=1}^m|f(X_j)|^{r}
|h(X_j)|^{r(p-1)}\Bigr)^{1/r}
\le
\|f\|_{\infty,X}^{t/r} \Bigl(\sum\limits_{j=1}^m|f(X_j)|^{r-t}|h(X_j)|^{pr-r}\Bigr)^{1/r}
\\
\le
\|f\|_{\infty,X}^{t/r}\Bigl(\sum\limits_{j=1}^m|f(X_j)|^{pr-t}\Bigr)^{\frac{r-t}{r(pr-t)}}
\Bigl(\sum\limits_{j=1}^m|h(X_j)|^{pr-t}\Bigr)^{\frac{pr-r}{r(pr-t)}}
\end{multline*}
and there is a number $c(p, r, t)$ such that for any $f, g\in D$
$$
\|f-g\|_{r, X; |h|^{r(p-1)}}
\le
c(p, r, t)\|f-g\|_{\infty,X}^{t/r}
\sup\limits_{u\in D}
\Bigl(\sum\limits_{j=1}^m|u(X_j)|^{pr-t}\Bigr)^{\frac{r-t}{r(pr-t)}}
\Bigl(\sum\limits_{j=1}^m|h(X_j)|^{pr-t}\Bigr)^{\frac{pr-r}{r(pr-t)}}
$$
which implies the second claim.

$3)$ Firstly, we note that for any $f, g\in B$ and for any $s\in(0, pr]$ one has
$$
\|f-g\|_{pr, X}^p
\le \|f-g\|_{\infty, X}^{s/r}
\Bigl(\sum_{j=1}^{m}|f(X_j) - g(X_j)|^{pr-s}\Bigr)^{1/r}
\le
2^{p}\|f-g\|_{\infty, X}^{s/r}\sup\limits_{u\in B}\Bigl(\sum_{j=1}^{m}|u(X_j)|^{pr-s}\Bigr)^{1/r}.
$$
If $p\ge\theta$,
by Theorem \ref{T-Conv}, one has
\begin{multline*}
\bigl[\gamma_{pr',p}(B, \|\cdot\|_{pr, X})\bigr]^p\le
\bigl[\gamma_{pr',\theta}(B, \|\cdot\|_{pr, X})\bigr]^p
\le
C_1(p, \theta, \zeta)\sup\limits_{k\ge 0}[2^{k/(pr')}e_k(B, \|\cdot\|_{pr, X})]^p
\\
\le
C_2(p, \theta, \zeta)\sup\limits_{u\in B}\Bigl(\sum_{j=1}^{m}|u(X_j)|^{pr-s}\Bigr)^{1/r}
\sup\limits_{k\ge 0}2^{k/r'}[e_k(B,\|\cdot\|_{\infty,X})]^{s/r}.
\end{multline*}
If $p\in(1, \theta)$,
by Theorem \ref{T-Conv-2}, one has
\begin{multline*}
\bigl[\gamma_{pr',p}(B, \|\cdot\|_{pr, X})\bigr]^p
\le
C_3(p, \theta, \zeta)\Bigl(\sum_{k\ge0}(2^{k/(pr')}
e_k(B, \|\cdot\|_{pr}))^{\frac{p\theta}{\theta-p}}\Bigr)^{\frac{\theta-p}{\theta}}
\\
\le
C_4(p, \theta, \zeta)\sup\limits_{u\in B}\Bigl(\sum_{j=1}^{m}|u(X_j)|^{pr-s}\Bigr)^{1/r}
\Bigl(\sum_{k\ge0}(2^{k/r'}e_k(B,\|\cdot\|_{\infty,X}^{s/r}))^{\frac{\theta}{\theta-p}}
\Bigr)^{\frac{\theta-p}{\theta}}.
\end{multline*}
The third claim is proved.
\end{proof}

The previous two lemmas imply the following conditional result
under the entropy numbers decay rate assumption.

\begin{theorem}\label{T0}
Let $p\in(1,\infty)$,  $\theta\ge 2$,  $\alpha\in (0, \infty)$,
and let $L$ be some subspace of~$L^p(\mu)\cap C(\Omega)$ for some Borel probability measure $\mu$
on a compact set $\Omega$.
Let $B\subset L$
be a
symmetric
$\theta$-convex (with constant $\zeta>0$) body.
Assume that for any fixed set of $m$ points $X=\{X_1, \ldots, X_m\}$
there is a constant $W_B(X)$ such that
$$
e_k(B, \|\cdot\|_{\infty, X})\le W_B(X)2^{-k/\alpha}.
$$

$1)$ Assume that $p\ge\alpha$.
Then there is a number $C:=C(p, \theta, \zeta, \alpha)$ such that
$$
\mathbb{E}\sup\limits_{f\in B}\Bigl|\frac{1}{m}\sum\limits_{j=1}^m|f(X_j)|^p - \|f\|_p^p\Bigr|
\le C\bigl(A +
A^{\frac{1}{\max\{\alpha, 2\}}}
(\sup\limits_{f\in B}\mathbb{E}|f(X_1)|^p)^{1-\frac{1}{\max\{\alpha, 2\}}}\bigr),
$$
where
$$
A= \frac{[\log m]^{\max\{\alpha, 2\}(1-\frac{1}{\theta})}}{m}
\mathbb{E}\bigl([W_B(X)]^\alpha
\sup\limits_{f\in B}\max\limits_{1\le j\le m}|f(X_j)|^{p-\alpha}\bigr).
$$

$2)$ Assume that $p\ge\max\{\alpha, 2\}$ and
assume that there is
a symmetric
$q$-convex (with a constant $\eta>0$) body $D\subset L$ such that
$B\subset D$. Assume that for any fixed discrete set of $m$ points $X=\{X_1, \ldots, X_m\}$
there is a constant $W_D(X)$ such that
$$
e_k(D, \|\cdot\|_{\infty, X})\le W_D(X)2^{-k/\beta}
$$
for some
$\beta\in[2, p]$. Then
there is a number $C:=C(p, \theta, \zeta, q, \eta, \alpha, \beta)$ such that
$$
\mathbb{E}\sup\limits_{f\in B}\Bigl|\frac{1}{m}\sum\limits_{j=1}^m|f(X_j)|^p - \|f\|_p^p\Bigr|
\le C\bigl(A_B+A_D +
(A_B+A_D)^{1/\beta}
(\sup\limits_{f\in B}\mathbb{E}|f(X_1)|^p)^{1-\frac{1}{\beta}}\bigr),
$$
where
$$
A_B= \frac{[\log m]^{\beta\max\{(1-\frac{p}{\theta}),0\}}}{m}
\mathbb{E}\bigl([W_B(X)]^\alpha
\sup\limits_{f\in B}\max\limits_{1\le j\le m}|f(X_j)|^{p-\alpha}\bigr)
$$
$$
A_D= \frac{[\log m]^{\beta(1-\frac{1}{q})}}{m}
\mathbb{E}\bigl([W_D(X)]^\beta
\sup\limits_{f\in B}\max\limits_{1\le j\le m}|f(X_j)|^{p-\beta}\bigr)
$$

$3)$ Assume that $p\ge\max\{\alpha, 2\}$ and assume that  there is a Euclidean ball $D\subset L$
such that $B\subset D$. Then
there is a constant $C:=C(p, \theta, \zeta, \alpha)$ such that
$$
\mathbb{E}\sup\limits_{f\in B}\Bigl|\frac{1}{m}\sum\limits_{j=1}^m|f(X_j)|^p - \|f\|_p^p\Bigr|
\le C\bigl(A +
A^{1/2}
(\sup\limits_{f\in B}\mathbb{E}|f(X_1)|^p)^{1/2}\bigr),
$$
where
\begin{multline*}
A=
\frac{1}{m}\mathbb{E}\bigl(\sup\limits_{f\in D}\max\limits_{1\le j\le m}|f(X_j)|^2
\sup\limits_{h\in B}\max\limits_{1\le j\le m}|h(X_j)|^{p-2}\bigr)
\\
+
\frac{[\log m]^{2\max\{1-\frac{p}{\theta}, 0\}}}{m}
\mathbb{E}\bigl([W_B(X)]^{\alpha}\sup\limits_{h\in B}
\max\limits_{1\le j\le m}|h(X_j)|^{p-\alpha}\bigr).
\end{multline*}
\end{theorem}

\begin{proof}
For $\theta, q\ge 2$,  $\alpha, \beta\in (0, \infty)$, $p\in[\max\{\alpha, \beta\},\infty)$,
consider any $\tau\ge\max\{\beta, 2\}$.
Let $r = \frac{\tau}{\tau-1}$, i.e. $r'=\frac{r}{r-1}=\tau$,
$t=\frac{\beta}{\tau-1}\le\frac{\tau}{\tau-1}=r$.
Applying Lemma~\ref{L1}(2) we get
\begin{multline*}
\sup\limits_{h\in B}\Bigl[\sum\limits_{k=0}^\infty\bigl(2^{k/r'}
e_k(D, \|\cdot\|_{r, X;|h|^{r(p-1)}})\bigr)^{\frac{q}{q-1}}\Bigr]^{\frac{q-1}{q}}
\\
\le
C_1(p, \tau, \beta)
\sup\limits_{f\in D}
\Bigl(\sum\limits_{j=1}^m|f(X_j)|^{\frac{p\tau-\beta}{\tau-1}}
\Bigr)^{\frac{\tau-1}{\tau}\cdot\frac{\tau-\beta}{p\tau-\beta}}
\sup\limits_{h\in B}
\Bigl(\sum\limits_{j=1}^m|h(X_j)|^{\frac{p\tau-\beta}{\tau-1}}
\Bigr)^{\frac{\tau-1}{\tau}\cdot\frac{\tau(p-1)}{p\tau-\beta}}
\\
\times
\Bigl[\sum\limits_{k=0}^\infty\bigl(2^{k/\tau}
[e_k(D, \|\cdot\|_{\infty, X})]^{\beta/\tau}\bigr)^{\frac{q}{q-1}}
\Bigr]^{\frac{q-1}{q}}.
\end{multline*}
We firstly note that
$$
\sup\limits_{f\in D}
\Bigl(\sum\limits_{j=1}^m|f(X_j)|^{\frac{p\tau-\beta}{\tau-1}}
\Bigr)^{\frac{\tau-1}{\tau}\cdot\frac{\tau-\beta}{p\tau-\beta}}
\le
\sup\limits_{f\in D}\|f\|_{\infty, X}^{\frac{p-\beta}{\tau}\cdot\frac{\tau-\beta}{p\tau-\beta}}
\sup\limits_{f\in D}
\Bigl(\sum\limits_{j=1}^m|f(X_j)|^p\Bigr)^{\frac{\tau-1}{\tau}\cdot\frac{\tau-\beta}{p\tau-\beta}}
$$
and
$$
\sup\limits_{h\in B}
\Bigl(\sum\limits_{j=1}^m|h(X_j)|^{\frac{p\tau-\beta}{\tau-1}}
\Bigr)^{\frac{\tau-1}{\tau}\cdot\frac{\tau(p-1)}{p\tau-\beta}}
\!\!\le\!
\sup\limits_{h\in B}\|h\|_{\infty, X}^{\frac{p-\beta}{\tau}\cdot\frac{\tau(p-1)}{p\tau-\beta}}
\!\!\sup\limits_{h\in B}
\Bigl(\sum\limits_{j=1}^m|h(X_j)|^p\Bigr)^{\frac{\tau-1}{\tau}\cdot\frac{\tau(p-1)}{p\tau-\beta}}
$$
Secondly, we note that the dimension $N_X$
of the linear space
$$L_X:=\{(f(X_1), \ldots, f(X_m))\colon f\in L\}$$
is not greater than $m$.
Thus, by the estimate \eqref{Eq1} for any $k>k_0:=[\log m]$
\begin{multline*}
e_k(D,\|\cdot\|_{\infty, X})\le 3\ 2^{2^{k_0}/N_X} e_{k_0}(D,\|\cdot\|_{\infty, X}) 2^{-2^k/N_X}
\\
\le
6e_{k_0}(D,\|\cdot\|_{\infty, X}) 2^{-2^k/m}\le
6\cdot 2^{1/\beta}W_D(X)m^{-1/\beta}2^{-2^k/m}
\end{multline*}
implying that
\begin{multline*}
\Bigl[\sum\limits_{k=0}^\infty\bigl(2^{k/\tau}
[e_k(B, \|\cdot\|_{\infty, X})]^{\beta/\tau}\bigr)^{\frac{q}{q-1}}
\Bigr]^{\frac{q-1}{q}}
\\
\le
C_2(\beta)[W_D(X)]^{\beta/\tau}\Bigl[\sum\limits_{k\le \log m} 1
+ m^{-\frac{q}{\tau(q-1)}}
\sum\limits_{k> \log m}
\bigl(2^{k/\beta}2^{-2^k/m}\bigr)^{\frac{\beta q}{\tau(q-1)}}
\Bigr]^{\frac{q-1}{q}}
\\
\le C_3(\beta, q, \tau)
[W_D(X)]^{\beta/\tau} [\log m]^{\frac{q-1}{q}},
\end{multline*}
where in the last inequality we have used the bound from Lemma \ref{L-bound}.

Let  $s=\frac{\alpha}{\tau-1}\le\frac{p}{\tau-1}\le\frac{p\tau}{\tau-1}=pr$.
By Lemma \ref{L1}(3),
for $p\ge\theta$ there is a positive number $C_4:=C_4(p, s, \theta, \zeta)$ such that
\begin{multline*}
\bigl[\gamma_{pr',p}(B, \|\cdot\|_{pr, X})\bigr]^p
\!\!\le\!
C_4\sup\limits_{f\in B}\Bigl(\sum_{j=1}^{m}|f(X_j)|^{\frac{p\tau-\alpha}{\tau-1}}
\Bigr)^{\frac{\tau-1}{\tau}}
\sup\limits_{k\ge 0}2^{k/\tau}[e_k(B,\|\cdot\|_{\infty, X})]^{\alpha/\tau}
\\
\le
C_4 \sup\limits_{f\in B}\|f\|_{\infty, X}^{\frac{p-\alpha}{\tau}}
\sup\limits_{f\in B}\Bigl(\sum_{j=1}^{m}|f(X_j)|^{p}
\Bigr)^{\frac{\tau-1}{\tau}}
[W_B(X)]^{\alpha/\tau}.
\end{multline*}

For $p\in(1, \theta)$, by the same Lemma \ref{L1}(3),
there is a number $C_4:=C_4(p, s, \theta, \zeta)$ such that
$$
\bigl[\gamma_{pr',p}(B, \|\cdot\|_{pr, X})\bigr]^p
\le
C_4\sup\limits_{f\in B}\Bigl(\sum_{j=1}^{m}|f(X_j)|^{\frac{p\tau-\alpha}{\tau-1}}
\Bigr)^{\frac{\tau-1}{\tau}}
\Bigl(\sum_{k\ge0}\bigl(2^{k/\tau}[e_k(B,\|\cdot\|_{\infty,X})]^{\alpha/\tau}
\bigr)^{\frac{\theta}{\theta-p}}\Bigr)^{\frac{\theta-p}{\theta}}.
$$
The first factor is bounded by
$$
\sup\limits_{f\in B}\|f\|_{\infty, X}^{\frac{p-\alpha}{\tau}}
\sup\limits_{f\in B}\Bigl(\sum_{j=1}^{m}|f(X_j)|^{p}
\Bigr)^{\frac{\tau-1}{\tau}}.
$$
To estimate the second factor we again use the inequality
\eqref{Eq1} which implies that for any $k>k_0:=[\log m]$ one has
$$
e_k(B,\|\cdot\|_{\infty, X})\le
6\cdot 2^{1/\alpha}W_B(X)m^{-1/\alpha}2^{-2^k/m}.
$$
Combining this bound with Lemma \ref{L-bound} we get
\begin{multline*}
\Bigl(\sum\limits_{k=0}^\infty\bigl(2^{k/\tau}
[e_k(B, \|\cdot\|_{\infty, X})]^{\alpha/\tau}\bigr)^{\frac{\theta}{\theta-p}}
\Bigr)^{\frac{\theta-p}{\theta}}
\\
\le
C_5(\alpha)[W_B(X)]^{\alpha/\tau}\Bigl(\sum\limits_{k\le \log m} 1
+ m^{-\frac{\theta}{\tau(\theta-p)}}
\sum\limits_{k> \log m}
\bigl(2^{k/\alpha}2^{-2^k/m}\bigr)^{\frac{\alpha\theta}{\tau(\theta-p)}}\Bigr)^{\frac{\theta-p}{\theta}}
\\
\le C_6(\alpha, \theta, \tau)
[W_B(X)]^{\alpha/\tau} [\log m]^{\frac{\theta-p}{\theta}}.
\end{multline*}

$1)$ We take $D=B$, $q=\theta$, $\beta=\alpha$, and any $\tau\in[\max\{\alpha, 2\}, \infty)$.
Since for $p\in(1, \theta)$ one has $1-\frac{p}{\theta}<1-\frac{1}{\theta}$,
Lemma~\ref{L0} and the above bounds imply that there is a
constant $C_7:=C_7(p, \theta, \zeta, \alpha, \tau)$ such that
$$
\mathbb{E}_\varepsilon\sup\limits_{f\in B}\Bigl|\sum\limits_{j=1}^m\varepsilon_j|f(X_j)|^p\Bigr|
\le
C_7[W_B(X)]^{\alpha/\tau}\sup\limits_{h\in B}\|h\|_{\infty, X}^{\frac{p-\alpha}{\tau}}
\sup\limits_{f\in B}\Bigl(\sum_{j=1}^{m}|f(X_j)|^{p}
\Bigr)^{1-\frac{1}{\tau}}
[\log m]^{1-\frac{1}{\theta}}.
$$
Lemma \ref{Lem-1} implies that there is a constant $C_8:=C_8(p, \theta, \zeta, \alpha, \tau)$
such that
$$
\mathbb{E}\sup\limits_{f\in B}\Bigl|\frac{1}{m}\sum\limits_{j=1}^m|f(X_j)|^p - \|f\|_p^p\Bigr|
\le C_8\bigl(A_\tau +
A_\tau^{1/\tau}(\sup\limits_{f\in B}\mathbb{E}|f(X_1)|^p)^{1-\frac{1}{\tau}}\bigr),
$$
where
$$
A_\tau= \frac{[\log m]^{\tau(1-\frac{1}{\theta})}}{m}\mathbb{E}\bigl([W_B(X)]^\alpha
\sup\limits_{f\in B}\max\limits_{1\le j\le m}|f(X_j)|^{p-\alpha}\bigr).
$$
Since the least power of logarithm is achieved for the minimal possible $\tau$
we take $\tau=\max\{\alpha, 2\}$
and get the first claim of the theorem.

$2)$ We take $\tau=\beta$. Then by Lemma~\ref{L0} and by the above bounds,
one can find a
constant $C_9:=C_9(p, \theta, \zeta, q, \eta, \alpha, \beta)$ such that
for $p\in(1, \theta)$ one has
\begin{multline*}
\mathbb{E}_\varepsilon\sup\limits_{f\in B}\Bigl|\sum\limits_{j=1}^m\varepsilon_j|f(X_j)|^p\Bigr|
\le
C_9
\sup\limits_{f\in B}\Bigl(\sum_{j=1}^{m}|f(X_j)|^{p}\Bigr)^{1-\frac{1}{\beta}}
\\
\times\Bigl(
\sup\limits_{h\in B}\|h\|_{\infty, X}^{\frac{p-\beta}{\beta}}
W_D(X)[\log m]^{1-\frac{1}{q}}+
[W_B(X)]^{\alpha/\beta}\sup\limits_{h\in B}\|h\|_{\infty, X}^{\frac{p-\alpha}{\beta}}
[\log m]^{1-\frac{p}{\theta}}\Bigr)
\end{multline*}
and for $p\in[\theta, \infty)$ one has
\begin{multline*}
\mathbb{E}_\varepsilon\sup\limits_{f\in B}\Bigl|\sum\limits_{j=1}^m\varepsilon_j|f(X_j)|^p\Bigr|
\le
C_9
\sup\limits_{f\in B}\Bigl(\sum_{j=1}^{m}|f(X_j)|^{p}\Bigr)^{1-\frac{1}{\beta}}
\\
\times\Bigl(
\sup\limits_{h\in B}\|h\|_{\infty, X}^{\frac{p-\beta}{\beta}}
W_D(X)[\log m]^{1-\frac{1}{q}}+
[W_B(X)]^{\alpha/\beta}\sup\limits_{h\in B}\|h\|_{\infty, X}^{\frac{p-\alpha}{\beta}}\Bigr).
\end{multline*}
Lemma \ref{Lem-1} now implies the second claim of the theorem.

$3)$ The Euclidean ball is $2$-convex.
We take $q=2$ and $\tau=2$. By Lemma~\ref{L1}(1), one has
\begin{multline*}
\sup\limits_{h\in B}\Bigl[\sum\limits_{k=0}^\infty\bigl(2^{k/2}
e_k(D, \|\cdot\|_{2, X;|h|^{2(p-1)}})\bigr)^2\Bigr]^{1/2}
\le
\bigl[\sup\limits_{f\in D}\|f\|_{\infty, X}\bigr]
\sup\limits_{h\in B}\Bigl(\sum\limits_{j=1}^{m}|h(X_j)|^{2(p-1)}\Bigr)^{1/2}
\\
\le \bigl[\sup\limits_{f\in D}\|f\|_{\infty, X}\bigr]\cdot
\bigl[\sup\limits_{h\in B}\|h\|_{\infty, X}^{\frac{p}{2}-1}\bigr]
\sup\limits_{h\in B}\Bigl(\sum\limits_{j=1}^{m}|h(X_j)|^{p}\Bigr)^{1/2}.
\end{multline*}
By Lemma \ref{L0} and by the above bounds, there is a number
$C_{10}:=C_{10}(p, \theta, \zeta, \alpha)$ such that
\begin{multline*}
\mathbb{E}_\varepsilon\sup\limits_{f\in B}\Bigl|\sum\limits_{j=1}^m\varepsilon_j|f(X_j)|^p\Bigr|
\le
C_{10}\sup\limits_{h\in B}\Bigl(\sum\limits_{j=1}^{m}|h(X_j)|^{p}\Bigr)^{1/2}
\\
\times\Bigl(
\sup\limits_{f\in D}\|f\|_{\infty, X}
\sup\limits_{h\in B}\|h\|_{\infty, X}^{\frac{p}{2}-1}
+
[W_B(X)]^{\alpha/2}\sup\limits_{h\in B}\|h\|_{\infty, X}^{\frac{p-\alpha}{2}}
[\log m]^{\max\{1-\frac{p}{\theta}, 0\}}
\Bigr).
\end{multline*}
Lemma \ref{Lem-1} now implies the third claim of the theorem.
\end{proof}

\begin{remark}\label{rem1}
{\rm
It follows from the proof that in the previous theorem
we need the assumptions
$$
e_k(B, \|\cdot\|_{\infty, X})\le W_B(X)2^{-k/\alpha}\quad
e_k(D, \|\cdot\|_{\infty, X})\le W_D(X)2^{-k/\beta}
$$
only for $k\le \log m$.
Actually, if $L$ is an $N$-dimensional subspace and $N\le m$
(as it is in the most cases we consider),
we need the above entropy numbers decay assumptions only for $k\le \log N$.
In that case in the above theorem all instances of $\log m$
should be replaced with $\log N$.
}
\end{remark}

\begin{remark}\label{rem-Dud-bound}{\rm
We note that under the assumptions of Theorem \ref{T0},
instead of Theorem \ref{T-VH}
one could use a simpler Dudley's entropy bound (see \cite[Proposition 2.2.10]{Tal})
to estimate the expectation of the
supremum of the Bernoulli process from Lemma \ref{L0}.
By this bound, applied with a
quasi-metric
$\varrho_\tau(f,g)$ for some fixed $\tau\ge\max\{\alpha, 2\}$
(see formula \eqref{eq-a-metric}), we have
\begin{multline*}
\mathbb{E}_\varepsilon\sup\limits_{f\in B}\Bigl|\sum\limits_{j=1}^m\varepsilon_j|f(X_j)|^p\Bigr|
\le C\sum\limits_{k=0}^\infty 2^{k/\tau}e_k(B, \varrho_\tau)
\\
\le
C_1\sup_{f,h\in B}
\Bigl(\sum\limits_{j=1}^m|f(X_j)|^{\tau'(1-\frac{\alpha}{\tau})}|h(X_j)|^{\tau'(p-1)}\Bigr)^{1/\tau'}
\sum\limits_{k=0}^\infty 2^{k/\tau}e_k(B, \|\cdot\|_{\infty, X}^{\alpha/\tau})
\\
\le\! C_2\!\!\sup_{f,h\in B}
\Bigl(\sum\limits_{j=1}^m|f(X_j)|^{\frac{p\tau-\alpha}{\tau-1}}
\Bigr)^{\frac{\tau-\alpha}{\tau'(p\tau-\alpha)}}\!
\Bigl(\sum\limits_{j=1}^m|h(X_j)|^{\frac{p\tau-\alpha}{\tau-1}}
\Bigr)^{\frac{p\tau-\tau}{\tau'(p\tau-\alpha)}}
W_B(X)^{\alpha/\tau}\log m
\\
=
C_2\sup_{f\in B}
\Bigl(\sum\limits_{j=1}^m|f(X_j)|^{\frac{p\tau-\alpha}{\tau-1}}\Bigr)^{1/\tau'}
W_B(X)^{\alpha/\tau}\log m,
\end{multline*}
where in the second estimate we have used
H$\ddot{o}$lder's inequality and
applied Lemma \ref{L-bound}.
Here numbers $C, C_1, C_2$ depend only on parameters
$\alpha, p, \tau$.
Assume that $p\ge \alpha$, then
$$
\Bigl(\sum\limits_{j=1}^m|f(X_j)|^{\frac{p\tau-\alpha}{\tau-1}}\Bigr)^{1/\tau'}
\le \|f\|_{\infty, X}^{\frac{p-\alpha}{\tau}}
\Bigl(\sum\limits_{j=1}^m|f(X_j)|^p\Bigr)^{1/\tau'}
$$
implying
$$
\mathbb{E}_\varepsilon\sup\limits_{f\in B}\Bigl|\sum\limits_{j=1}^m\varepsilon_j|f(X_j)|^p\Bigr|
\le\!
C_2\!\sup_{h\in B}\|h\|_{\infty, X}^{\frac{p-\alpha}{\tau}}
\sup_{f\in B}
\Bigl(\sum\limits_{j=1}^m|f(X_j)|^p\Bigr)^{1-\frac{1}{\tau}}\!
W_B(X)^{\alpha/\tau}\log m.
$$
By Lemma \ref{Lem-1}, taking the minimal possible $\tau=\max\{\alpha, 2\}$
(to minimize the power of logarithm), we get
$$
\mathbb{E}\sup\limits_{f\in B}\Bigl|\frac{1}{m}\sum\limits_{j=1}^m|f(X_j)|^p - \|f\|_p^p\Bigr|
\le C\bigl(A +
A^{\frac{1}{\max\{\alpha, 2\}}}
(\sup\limits_{f\in B}\mathbb{E}|f(X_1)|^p)^{1-\frac{1}{\max\{\alpha, 2\}}}\bigr),
$$
where
$$
A= \frac{[\log m]^{\max\{\alpha, 2\}}}{m}
\mathbb{E}\bigl([W_B(X)]^\alpha
\sup\limits_{f\in B}\max\limits_{1\le j\le m}|f(X_j)|^{p-\alpha}\bigr).
$$
This bound is valid for any convex set $B\subset L$ (not necessarily $\theta$-convex),
but omitting the additional information about $\theta$-convexity
we lose $[\log m]^{\frac{\max\{\alpha, 2\}}{\theta}}$ factor.
}
\end{remark}

Theorem \ref{T0} already provides the
following conditional result for the problem of sampling discretization.

\begin{corollary}
Let $p\in(1,\infty)$
and let $L$ be any $N$-dimensional subspace of~$L^p(\mu)\cap C(\Omega)$
for some Borel probability measure $\mu$
on a compact set $\Omega$.
Let $B_p(L):=\{f\in L\colon \|f\|_p\le1\}$.
Assume that for any fixed discrete set of $m\ge N$ points $X=\{X_1, \ldots, X_m\}$
there is a constant $W(X)$ such that
$$
e_k(B_p(L), \|\cdot\|_{\infty, X})\le W(X)2^{-k/p}.
$$
Then
there is a number $C:=C(p)$ such that
$$
\mathbb{E}\sup\limits_{f\in B_p(L)}\Bigl|\frac{1}{m}\sum\limits_{j=1}^m|f(X_j)|^p - \|f\|_p^p\Bigr|
\le C\bigl(A + A^{\frac{1}{\max\{p, 2\}}}\bigr),
$$
where
$$
A= \frac{[\log N]^{\max\{p, 2\}-1}}{m}
\mathbb{E}\bigl([W(X)]^p\bigr).
$$
In particular, there is a large enough constant $c(p)$ such that
for every $\delta\in(0, 1)$, for every $\varepsilon\in(0, 1)$
and for every $m\ge N$ such that
$$m\ge c(p)(\delta\varepsilon)^{-\max\{p, 2\}}
\mathbb{E}\bigl([W(X)]^p\bigr)[\log N]^{\max\{p, 2\}-1}$$
one has
$$
(1-\varepsilon)\|f\|_p^p\le \sum_{j=1}^{m}|f(X_j)|^p\le (1+\varepsilon)\|f\|_p^p\quad \forall f\in L
$$
with probability greater than $1-\delta$.
\end{corollary}

\begin{proof}
The first part follows from Theorem \ref{T0}(1) and Remark \ref{rem1}
since the $L^p$-norm is $\max\{p, 2\}$-convex with some constant $\zeta(p)$.
The second part is just the application of Chebyshev's inequality.
\end{proof}

\begin{remark}
{\rm
We note that such conditional result is already applicable in many situations,
since in many cases one can independently obtain bounds for the entropy numbers
even with respect to the uniform norm $\|\cdot\|_\infty$ in place of discretized uniform norm
$\|\cdot\|_{\infty, X}$.
For example, this is the case for the so called hyperbolic cross trigonometric polynomials
(see \cite{Tem2} and \cite{Tem6}).
}
\end{remark}

To obtain general results without explicit assumptions
on the entropy numbers
one needs to use general bounds for the entropy numbers.
We will do in the next section.

\section{Discretization under the Nikolskii-type inequality assumption}

First of all, there is a bound for the entropy numbers of a general $\theta$-convex set
with respect to the discretized uniform norm $\|\cdot\|_{\infty, X}$
(see \cite[Lemma~16.5.4]{Tal} and \cite{Tal-Sect}).
We recall this bound in the form it is stated in \cite{Tal} and then
reformulate it for our case.
\begin{lemma}[see Lemma~16.5.4 in \cite{Tal}]\label{Lem-Tal}
Let $(E,\|\cdot\|_E)$ be a Banach space and let the norm
$\|\cdot\|_{E^*}$ in the dual space $E^*$
be $\theta$-convex with some constant $\zeta>0$ for some $\theta\ge2$.
For a fixed set of vectors $\Phi:=\{\varphi_1, \ldots, \varphi_m\}$, consider
a (semi)norm $\|\varphi^*\|_{\infty, \Phi}:=\max\limits_{1\le j\le m}|\varphi^*(\varphi_j)|$ on~$E^*$.
Then for some number $C:=C(\theta,\zeta)$, which depends only on $\theta$ and $\zeta$, one has
$$
e_k(B_*,\|\cdot\|_{\infty, \Phi})\le
C\bigl[\max\limits_{1\le j\le m}\|\varphi_j\|_{E}\bigr] 2^{-k/\theta}[\log m]^{1/\theta},
$$
where $B_*:=\{\varphi^*\in E^*\colon \|\varphi^*\|_{E^*}\le 1\}$.
\end{lemma}

\begin{corollary}\label{ent-est}
Let $L$ be a linear space of functions defined on some set $X=\{X_1, \ldots, X_m\}$
and let $B\subset L$
be a $\theta$-convex body with some constant $\zeta>0$.
Then there is a constant $C:=C(\theta, \zeta)$ such that
$$
e_k(B,\|\cdot\|_{\infty, X})\le
C\bigl[\max\limits_{1\le j\le m}\sup\limits_{f\in B}|f(X_j)|\bigr] 2^{-k/\theta}[\log m]^{1/\theta}.
$$
\end{corollary}

\begin{proof}
We note that $B$ is the unit ball of some $\theta$-convex norm $\|\cdot\|_L$.
Let $E$ be the dual space (with respect to this norm) to $L$, i.e. $E=L^*$.
Then $L=E^*$ ($L$ is finite dimensional) and functionals
$\varphi_j(f):=f(X_j)$ are elements of $L^*=E$.
Thus, we take $\Phi:=\{\varphi_1, \ldots, \varphi_m\} \subset E$ and
by the above lemma one has
$$
e_k(B,\|\cdot\|_{\infty, \Phi})\le
K\bigl[\max\limits_{1\le j\le m}\|\varphi_j\|_{E}\bigr] 2^{-k/\theta}[\log m]^{1/\theta}.
$$
It remains to notice that
$\|f\|_{\infty, \Phi} = \max\limits_{1\le j\le m}|f(X_j)|$ for each $f\in L$
and that $\|\varphi_j\|_E = \sup\limits_{f\in B}|f(X_j)|$.
The corollary is proved.
\end{proof}

\begin{remark}
{\rm
It is interesting to note that
one can obtain Lemma \ref{Lem-Tal} from the greedy approximation theory.
Without loss of generality, we assume that $\|\varphi_j\|_E=1$, $\forall j\in\{1,\ldots, m\}$.
Let $U$ be a convex hull of $\pm x_1,\ldots,\pm x_m$.
The first step is the same as in Talagrand's work \cite[Lemma 3.3]{Tal-Sect}:
by iterations of Proposition~$2$ from \cite{BPST-J}
the desired estimate follows from the bound
$$e_k(U, \|\cdot\|_E)\le K(p,\eta) 2^{-k/\theta}[\log m]^{1/\theta}.$$
And now this bound can be deduced from the bound for the best $n$-term approximation:
let $\mathcal{D}=\{y_j\}$ be a set of $r$ points in $E$, then
$$
\sigma_n(U,\mathcal{D}):=\sup\limits_{y\in U}\inf\limits_{\{c_j\}, |\Lambda|=n}\|y-\sum_{j\in\Lambda}c_jy_j\|_E.
$$
It is known (see \cite[Theorem 7.4.3]{TemBook} and \cite[Theorem 3.1]{Tem5}) that
$e_k(U, \|\cdot\|)\le C(\omega)A[\log 2r]^{\omega}2^{-\omega k}$ for every $k\le \log r$
provided that there is a system $\mathcal{D}$ of $r$ elements such that
$\sigma_n(U,\mathcal{D})\le A n^{-\omega}$ for every $n\le r$.
We note that the unit ball in the dual space $E$ is $\theta' = \frac{\theta}{\theta-1}$-smooth.
Now taking $\mathcal{D}=\{\pm x_1,\ldots, \pm x_m\}$ and
applying Weak Chebyshev Greedy Algorithm (see \cite[Section 6.2]{TemBook2}),
we get
$\sigma_n(U,D)\le C(\theta,\eta) n^{-1/\theta}$ (see \cite[Theorem 6.8]{TemBook2}).
Thus, for $k\le\log m$, one has
$$e_k(U, \|\cdot\|)\le C_1(\theta,\eta)[\log 4m]^{1/\theta}2^{-k/\theta}\le
C_2(p,\eta)[\log m]^{1/\theta}2^{-k/\theta}.$$
See more on this observation in \cite{Tem7}.
}
\end{remark}

Corollary \ref{ent-est} combined with Theorem \ref{T0}(3)
already allows to improve the main
result of \cite{GueRud} and combined with Theorem \ref{T0}(1)
provides several results for the sampling discretization problem.

\begin{corollary}\label{cor1}
Let $\theta\ge 2$, $p\in[\theta,\infty)$,
and let $L$ be a subspace of~$L^p(\mu)\cap C(\Omega)$
for some Borel probability measure $\mu$
on a compact set $\Omega$.
Let $B\subset L$
be a
symmetric
$\theta$-convex (with a constant $\zeta>0$) body
and assume that there is an Euclidean ball $D\subset L$
such that $B\subset D$.
Then
there is a constant $C:=C(p, \theta, \zeta)$ such that
$$
\mathbb{E}\sup\limits_{f\in B}\Bigl|\frac{1}{m}\sum\limits_{j=1}^m|f(X_j)|^p - \|f\|_p^p\Bigr|
\le C\bigl(A + A^{1/2}
(\sup\limits_{f\in B}\mathbb{E}|f(X_1)|^p)^{1/2}\bigr),
$$
where
$$
A=
\frac{1}{m}\mathbb{E}\bigl(\sup\limits_{f\in D}\max\limits_{1\le j\le m}|f(X_j)|^2
\sup\limits_{h\in B}\max\limits_{1\le j\le m}|h(X_j)|^{p-2}\bigr)
+
\frac{\log m}{m}\mathbb{E}\bigl(\sup\limits_{h\in B}
\max\limits_{1\le j\le m}|h(X_j)|^{p}\bigr).
$$
\end{corollary}

In particular, we get the following result on discretization under the
$(\infty, 2)$ Nikolskii-type
inequality assumption.

\begin{corollary}\label{cor2}
Let $p\in[2,\infty)$ and let $\mu$ be a probability Borel measure
on a compact set $\Omega$.
There is a number $C:=C(p)$, dependent only on $p$,
such that, if $L$ is an
$N$-dimensional subspace of~$L^p(\mu)\cap C(\Omega)$
such that
$$
\|f\|_\infty\le MN^{1/2}\|f\|_2\quad \forall f\in L,
$$
then
$$
\mathbb{E}\sup\limits_{f\in B_p(L)}\Bigl|\frac{1}{m}\sum\limits_{j=1}^m|f(X_j)|^p - \|f\|_p^p\Bigr|
\le C\Bigl(\frac{\log m}{m} M^pN^{p/2} + \Bigl[\frac{\log m}{m} M^pN^{p/2}\Bigr]^{1/2}\Bigr),
$$
where $B_p(L):=\{f\in L\colon \|f\|_p\le 1\}$.
In particular, for every $\varepsilon\in(0, 1)$ and for every $\delta\in(0, 1)$
there is a big enough constant $c:=c(p, \varepsilon, \delta)$ such that
for every $m\ge cM^pN^{p/2}\log(4M^2N)$
one has
$$
(1-\varepsilon)\|f\|_p^p\le \sum_{j=1}^{m}|f(X_j)|^p\le (1+\varepsilon)\|f\|_p^p\quad \forall f\in L
$$
with probability greater than $1-\delta$ for any such subspace $L$.
\end{corollary}

\begin{proof}
Since for $p\ge 2$ the ball $B_p(L)$ is $p$-convex (with some constant $\zeta(p)$)
and $B_p(L)\subset B_2(L)$,
we can apply the previous corollary with $B=B_p(L)$
and with Euclidean ball $D=B_2(L)$.
We also note that
\begin{multline*}
\sup\limits_{h\in B}
\max\limits_{1\le j\le m}|h(X_j)|^{p}
\le
\sup\limits_{f\in D}\max\limits_{1\le j\le m}|f(X_j)|^2
\sup\limits_{h\in B}\max\limits_{1\le j\le m}|h(X_j)|^{p-2}
\le \sup\limits_{f\in D}\max\limits_{1\le j\le m}|f(X_j)|^p = M^pN^{p/2}
\end{multline*}
Thus, the first part of the assertion follows from Corollary \ref{cor1}.
The part concerning the discretization follows from Chebyshev's inequality,
since $M$ is always greater than or equal~to~$1$.
\end{proof}

\begin{remark}\label{rem-weights}
{\rm
We note that Corollary \ref{cor2} combined with
Lewis' change of density theorem (see \cite{Lew} or \cite{SZ})
implies
that for every $p\ge2$ and for every $\varepsilon\in(0, 1)$
there is a big enough constant $c:=c(p, \varepsilon)$
such that for every $N$-dimensional subspace $L$
of $L^p(\mu)\cap C(\Omega)$ and
for each $m\ge c N^{p/2}\log N$ there are points $X_1, \ldots, X_m$
and there are
positive weights
$\lambda_1, \ldots, \lambda_m$ such that
$$
(1-\varepsilon)\|f\|_p^p\le \sum_{j=1}^{m}\lambda_j|f(X_j)|^p\le (1+\varepsilon)\|f\|_p^p\quad \forall f\in L.
$$
The proof is the same as the proof of Theorem 2.3 in \cite{5-2}.
}
\end{remark}

We note that Corollary \ref{cor2} gives only a power dependence on the dimension~$N$
for the number of discretizing points.
Thus, we seek conditions on $L$ under which one can guarantee linear or almost linear dependence on dimension for the number of points sufficient for discretization.
For this purpose we combine Theorem \ref{T0}(1) with the estimate for the entropy numbers
from Corollary~\ref{ent-est}.

\begin{corollary}\label{cor3}
Let $\theta\ge 2$, $p\in[\theta,\infty)$,
and let $L$ be a subspace of~$L^p(\mu)\cap C(\Omega)$
for some Borel probability measure $\mu$
on a compact set $\Omega$.
Let $B\subset L$
be a
symmetric
$\theta$-convex (with a constant $\zeta>0$) body.
Then
there is a constant $C:=C(p, \theta, \zeta)$ such that
$$
\mathbb{E}\sup\limits_{f\in B}\Bigl|\frac{1}{m}\sum\limits_{j=1}^m|f(X_j)|^p - \|f\|_p^p\Bigr|
\le C\bigl(A +
A^{\frac{1}{\theta}}
(\sup\limits_{f\in B}\mathbb{E}|f(X_1)|^p)^{1-\frac{1}{\theta}}\bigr),
$$
where
$$
A= \frac{[\log m]^{\theta}}{m}
\mathbb{E}\bigl(\sup\limits_{f\in B}\max\limits_{1\le j\le m}|f(X_j)|^{p}\bigr).
$$
\end{corollary}

Since the $L^p$-norm is $p$-convex with some constant $\zeta(p)$
for $p\ge 2$, the above
corollary implies the following result on sampling discretization under the
$(\infty, p)$ Nikolskii-type
inequality assumption.

\begin{corollary}\label{cor4}
Let $p\in[2,\infty)$ and let $\mu$ be a probability Borel measure
on a compact set $\Omega$.
There is a number $C:=C(p)$, dependent only on $p$,
such that,
if $L$ is an
$N$-dimensional subspace of~$L^p(\mu)\cap C(\Omega)$
such that
$$
\|f\|_\infty\le MN^{1/p}\|f\|_p\quad \forall f\in L,
$$
then
$$
\mathbb{E}\sup\limits_{f\in B_p(L)}\Bigl|\frac{1}{m}\sum\limits_{j=1}^m|f(X_j)|^p - \|f\|_p^p\Bigr|
\le C\Bigl(\frac{[\log m]^p}{m} M^pN + \Bigl[\frac{[\log m]^p}{m} M^pN\Bigr]^{1/p}\Bigr),
$$
where $B_p(L):=\{f\in L\colon \|f\|_p\le 1\}$.
In particular, for every $\varepsilon\in(0, 1)$ and for every $\delta\in(0, 1)$
there is a big enough constant $c:=c(p, \varepsilon, \delta)$ such that
for every $m\ge cM^pN[\log(4M^pN)]^p$
one has
$$
(1-\varepsilon)\|f\|_p^p\le \sum_{j=1}^{m}|f(X_j)|^p\le (1+\varepsilon)\|f\|_p^p\quad \forall f\in L
$$
with probability greater than $1-\delta$ for any such subspace $L$.
\end{corollary}

If the $\theta$-convex body $B$ is contained in another $q$-convex body $D$,
we can combine Theorem \ref{T0}(2) and entropy numbers bound from Corollary \ref{ent-est}
and get the following analog of Corollary~\ref{cor1}.

\begin{corollary}\label{cor5}
Let $\theta\ge 2$, $q\ge2$, $p\in[\max\{\theta, q\},\infty)$,
and let $L$ be a subspace of~$L^p(\mu)\cap C(\Omega)$
for some Borel probability measure $\mu$
on a compact set $\Omega$.
Let $D\subset L$
be a
symmetric
$q$-convex (with a constant $\eta>0$) body
and let $B\subset D$
be a
symmetric
$\theta$-convex (with a constant $\zeta>0$) body.
Then
there is a constant $C:=C(p, \theta, \zeta, q, \eta)$ such that
$$
\mathbb{E}\sup\limits_{f\in B}\Bigl|\frac{1}{m}\sum\limits_{j=1}^m|f(X_j)|^p - \|f\|_p^p\Bigr|
\le C\bigl(A + A^{\frac{1}{q}}
(\sup\limits_{f\in B}\mathbb{E}|f(X_1)|^p)^{1-\frac{1}{q}}\bigr),
$$
where
$$
A=
\frac{[\log m]^q}{m}\mathbb{E}\bigl(\sup\limits_{f\in D}\max\limits_{1\le j\le m}|f(X_j)|^q
\sup\limits_{h\in B}\max\limits_{1\le j\le m}|h(X_j)|^{p-q}\bigr)
+
\frac{\log m}{m}\mathbb{E}\bigl(\sup\limits_{h\in B}
\max\limits_{1\le j\le m}|h(X_j)|^{p}\bigr).
$$
\end{corollary}

We note that all the above results are not applicable in the case $p\in(1, 2)$
and that is why we need to use better bounds for the entropy numbers of the $L^p$
balls for $p\in(1, 2)$.

\begin{lemma}\label{ent-est-2}
Let $p\in(1,2)$ and let $\mu$ be a probability Borel measure
on a compact set $\Omega$.
There is a constant $C:=C(p)$ such that,
if $L$ is an $N$-dimensional subspace of~$L^p(\mu)\cap C(\Omega)$
such that
$$
\|f\|_\infty\le M\|f\|_2\quad \forall f\in L
$$
for some number $M\ge 2$,
then
for any fixed set of $m$ points $X=\{X_1, \ldots, X_m\}$
one has
$$
e_k(B_p(L), \|\cdot\|_{\infty,X})
\le C[\log m]^{1/2} [\log M]^{\frac{1}{p}-\frac{1}{2}}M^{2/p}
2^{-k/p},
$$
where $B_p(L)=\{f\in L\colon \|f\|_p\le1\}$.
\end{lemma}

The proof of this lemma is actually very similar to the proof of
\cite[Proposition 16.8.6]{Tal} and we present it in Appendix C.

Since the unit ball in $L^p$-norm
is $2$-convex for $p\in(1, 2)$
we now can combine Lemma \ref{ent-est-2} and Theorem \ref{T0}(1)
and obtain the following result on sampling discretization.

\begin{corollary}\label{cor6}
Let $p\in(1, 2)$ and let $\mu$ be a probability Borel measure
on a compact set $\Omega$.
There is a constant $C:=C(p)$ such that,
if $L$ is an $N$-dimensional subspace of~$L^p(\mu)\cap C(\Omega)$
such that
$$
\|f\|_\infty\le MN^{1/2}\|f\|_2\quad \forall f\in L,
$$
then
\begin{multline*}
\mathbb{E}\sup\limits_{f\in B_p(L)}\Bigl|\frac{1}{m}\sum\limits_{j=1}^m|f(X_j)|^p - \|f\|_p^p\Bigr|
\\
\le C\Bigl(\frac{[\log m]^{1+\frac{p}{2}}[\log 4M^2N]^{1-\frac{p}{2}}}{m} M^2N +
\Bigl[\frac{[\log m]^{1+\frac{p}{2}}[\log 4M^2N]^{1-\frac{p}{2}}}{m} M^2N\Bigr]^{1/2}\Bigr),
\end{multline*}
where $B_p(L):=\{f\in L\colon \|f\|_p\le 1\}$.
In particular, for every $\varepsilon\in(0, 1)$ and for every $\delta\in(0, 1)$
there is a big enough constant $c:=c(p, \varepsilon, \delta)$ such that
for every $m\ge cM^2N[\log(4M^2N)]^2$
one has
$$
(1-\varepsilon)\|f\|_p^p\le \sum_{j=1}^{m}|f(X_j)|^p\le (1+\varepsilon)\|f\|_p^p\quad \forall f\in L
$$
with probability greater than $1-\delta$ for any such subspace $L$.
\end{corollary}

\begin{remark}\label{rem-weights-2}
{\rm
Similarly to Remark \ref{rem-weights},
the combination of Corollary \ref{cor6} and
Lewis' change of density theorem (see \cite{Lew} or \cite{SZ})
implies
that for every $p\in(1, 2)$ and for every $\varepsilon\in(0, 1)$
there is a big enough constant $c:=c(p, \varepsilon)$
such that for every $N$-dimensional subspace $L$
of $L^p(\mu)\cap C(\Omega)$ and
for each $m\ge c N[\log N]^2$ there are points $X_1, \ldots, X_m$
and there are
positive weights
$\lambda_1, \ldots, \lambda_m$ such that
$$
(1-\varepsilon)\|f\|_p^p\le \sum_{j=1}^{m}\lambda_j|f(X_j)|^p\le (1+\varepsilon)\|f\|_p^p\quad \forall f\in L.
$$
The proof again is the same as the proof of Theorem 2.3 in \cite{5-2}.
}
\end{remark}

\section{Appendix A: the proof of Theorem \ref{T-VH}}

We again stress that the proof of Theorem \ref{T-VH} heavily follows the proof of
\cite[Theorem 7.3]{VH} and is presented here
only for readers' convenience.

We first recall the claim of the theorem.

{\bf Theorem \ref{T-VH}.}
{\it
Let $q\ge2$, $p>1$, $\alpha>0$.
Let $L$ be a linear space and let $D\subset L$
be a symmetric $q$-convex (with constant $\eta>0$)
body.
Let $\varrho$ be a quasi-metric on $L$
such that
$$
\varrho(f,g)\le R\bigl(\varrho(f,h)+\varrho(h,g)\bigr);\quad \varrho\Bigl(f, \frac{f+g}{2}\Bigr)\le \varkappa \varrho(f,g)
$$
for all $f,g,h\in L$, for some constants $R,\varkappa>0$.
Assume that there is a metric $d$ on $L$
and for each $h\in L$ there is a norm $\|\cdot\|_h$ on $L$ such that
$$
c_1 d(f,g)^p\le \varrho(f,g)\le c_2\bigl(\|f-g\|_h+d(f,g)(d(f,h)^{p-1}+d(h,g)^{p-1})\bigr)
$$
for some numbers $c_1, c_2>0$.
Then there is a number $C:=C(q,p,\alpha, R,\varkappa, c_1, c_2)$
such that for any $B\subset D$ one has
$$
\gamma_{\alpha,1}(B, \varrho)\le
C\Bigl(
\eta^{-1/q}\Bigl[\sup\limits_{h\in B}
\sum\limits_{k=0}^\infty\bigl(2^{k/\alpha}e_k(D, \|\cdot\|_h)\bigr)^{\frac{q}{q-1}}\Bigr]^{\frac{q-1}{q}}
+
\bigl[\gamma_{\alpha p,p}(B, d)\bigr]^p
\Bigr).
$$
}

Since $D\subset L$
is a symmetric $q$-convex with constant $\eta>0$ body,
then it is a unit ball with respect to some
$q$-convex with constant $\eta>0$ norm $\|\cdot\|$, i.e.
$D=\{f\in L\colon \|f\|\le 1\}$ and
$$
\Bigl\|\frac{f+g}{2}\Bigr\|\le \max(\|f\|,\|g\|)-\eta\|f-g\|^q
$$
for any $f, g$ with $\|f\|\le1, \|g\|\le1$.

We recall the main tools from \cite{VH} concerning chaining through interpolation.
Let
$$
K(t, f):=\inf\limits_{g\in L}(\|g\|+t\varrho(f,g))
$$
and let $\pi_t(f)$ be any minimizer.

The following contraction principle is formulated and proved in Theorem 3.1 in \cite{VH}.

\begin{theorem}\label{Contraction}
Assume there are functions $s_k(f)\ge0$ and a number $a>0$ such that
$$
e_k(A, \varrho)\le a\, {\rm diam}(A,\varrho) + \sup\limits_{f\in A}s_k(f)
$$
for every $k\in \mathbb{N}\cup\{0\}$ and every set $A\subset B$. Then
$$
\gamma_{\alpha,r}(B,\varrho)\le C(\alpha)\Bigl(a\, \gamma_{\alpha,r}(B,\varrho) +
\Bigl[\sup_{f\in B}\sum\limits_{k\ge0}\bigl(2^{k/\alpha}s_k(f)\bigr)^r\Bigr]^{1/r}\Bigr).
$$
\end{theorem}

The following theorem is Lemma 4.5 in \cite{VH}.

\begin{theorem}\label{Interpolation}
For every $a>0$ one has
$$
\sup_{f\in B}\sum\limits_{k\ge0}2^{k/\alpha}\varrho(f, \pi_{a2^{k/\alpha}}(f))\le
C(\alpha)a^{-1}\sup\limits_{f\in B}\|f\|.
$$
\end{theorem}

Throughout this section the expression $V\lesssim W$ means
that there exists a positive number $C:=C(q,p,\alpha, R,\varkappa, c_1, c_2)$
such that $V\le C W$.

\begin{lemma}\label{Lem-A}
For any $t>0$ and for any $A\subset B\subset D$ one has
$$
{\rm diam}(A_t, \|\cdot\|)\le
c(\varkappa, R, q)\Bigl(\frac{t}{\eta}\Bigr)^{1/q}\Bigl({\rm diam}(A, \varrho) +
\sup\limits_{h\in A}\varrho(h, \pi_t(h))\Bigr)^{1/q}
$$
where
$A_t:=\{\pi_t(h)\colon h\in A\}$.
\end{lemma}

\begin{proof}
We note that
$$
\|\pi_t(f)\|\le K(t, f)\le \|u\|+tR(\varrho(f,\pi_t(f))+\varrho(\pi_t(f),u))
$$
for any $u\in L$. Thus, for fixed $f,g\in A$ we take $u=\frac{1}{2}\bigl(\pi_t(f)+\pi_t(g)\bigr)$
and obtain
$$
\max(\|\pi_t(f)\|,\|\pi_t(g)\|)\le \Bigl\|\frac{\pi_t(f)+\pi_t(g)}{2}\Bigr\| +
tR\sup\limits_{h\in A}\varrho(h, \pi_t(h)) +
tR\varkappa \varrho(\pi_t(f), \pi_t(g)).
$$
By the definition of $q$-convexity, we get
\begin{multline*}
\eta\|\pi_t(f) - \pi_t(g)\|^q\le tR\sup\limits_{h\in A}\varrho(h, \pi_t(h)) + tR\varkappa \varrho(\pi_t(f), \pi_t(g))
\\
\le
tR^3\varkappa \varrho(f, g) + t(R+\varkappa R^2+\varkappa R^3)\sup\limits_{h\in A}\varrho(h, \pi_t(h)).
\end{multline*}
This bound implies
the statement of the lemma.
\end{proof}

\begin{remark}{\rm
The lemma actually means that
the set $A_t$ is contained in some
ball of radius
$c(\varkappa, R, q)\bigl(\frac{t}{\eta}\bigr)^{1/q}\Bigl({\rm diam}(A, \varrho) +
\sup\limits_{h\in A}\varrho(h, \pi_t(h))\Bigr)^{1/q}$
with respect to the norm $\|\cdot\|$.
}
\end{remark}

\begin{lemma}\label{Lem-B}
Let $(\mathcal{F}_n)$ be an admissible sequence of $B$ and $a,b>0$.
Then
$$
e_k(A,\varrho)\lesssim b\, {\rm diam}(A,\varrho) + \sup\limits_{f\in A}s_k(f)
$$
for every $k\ge1$ and every $A\subset B$
where
\begin{multline*}
s_k(f)=(b+1)\varrho(f, \pi_{a2^{k/\alpha}}(f))
+
\Bigl(\frac{a2^{k/\alpha}}{b\eta}\Bigr)^{1/(q-1)}\bigl(e_{k-1}(D, \|\cdot\|_f)\bigr)^{q/(q-1)}
+
\bigl({\rm diam}(F_{k-1}(f), d)\bigr)^p.
\end{multline*}
\end{lemma}

\begin{proof}
For any $F\subset B$ let
$$
A^F_{a2^{k/\alpha}}:=\{\pi_{a2^{k/\alpha}}(f)\colon f\in A\cap F\},
$$
let $h_F$ be any point in $A\cap F$ and let $T_{k-1}^F\subset A^F_{a2^{k/\alpha}}$ be any net such that
$|T_{k-1}^F|\le 2^{2^{k-1}}$ and
$$
\sup\limits_{u\in A^F_{a2^{k/\alpha}}}\inf_{g\in T_{k-1}^F} \|u-g\|_{h_F}\le
4 e_{k-1}(A^F_{a2^{k/\alpha}}, \|\cdot\|_{h_F}).
$$
Let $T_k:=\bigcup\limits_{F\in\mathcal{F}_{k-1}}T^F_{k-1}$. Note that $|T_k|\le 2^{2^k}$.
We now show that
$$
\sup\limits_{f\in A}\inf_{v\in T_k}\varrho(f,v)
\lesssim b\, {\rm diam}(A, \varrho) + \sup\limits_{f\in A}s_k(f).
$$
Let $f\in A$ and let $g\in T_{k-1}^{F_{k-1}(f)}$ be such that
$$
\|\pi_{a2^{k/\alpha}}(f) - g\|_{h_{F_{k-1}(f)}}\le
4e_{k-1}(A^{F_{k-1}(f)}_{a2^{k/\alpha}}, \|\cdot\|_{h_{F_{k-1}(f)}}).
$$
We have
\begin{multline*}
\inf_{v\in T_k}\varrho(f,v)\le R\varrho(f, \pi_{a2^{k/\alpha}}(f)) + R\varrho(\pi_{a2^{k/\alpha}}(f), g)
\lesssim
\varrho(f, \pi_{a2^{k/\alpha}}(f)) + \|\pi_{a2^{k/\alpha}}(f) - g\|_{h_{F_{k-1}(f)}}
\\
+
d(\pi_{a2^{k/\alpha}}(f),g)\bigl(d(\pi_{a2^{k/\alpha}}(f),h_{F_{k-1}(f)})^{p-1} + d(h_{F_{k-1}(f)},g)^{p-1}\bigr).
\end{multline*}
Since $g\in T_{k-1}^{F_{k-1}(f)}\subset A^{F_{k-1}(f)}_{a2^{k/\alpha}}$ there is
an element
$f'\in A\cap F_{k-1}(f)$ such that $g=\pi_{a2^{k/\alpha}}(f')$.
Thus,
\begin{multline*}
d(\pi_{a2^{k/\alpha}}(f),g)
\le
d(\pi_{a2^{k/\alpha}}(f),f) + d(f,f')+d(f',\pi_{a2^{k/\alpha}}(f'))
\\
\le 2\sup\limits_{h\in A}d(\pi_{a2^{k/\alpha}}(h),h) + {\rm diam}(F_{k-1}(f), d)
\lesssim
\sup\limits_{h\in A}\bigl(\varrho(\pi_{a2^{k/\alpha}}(h),h)\bigr)^{1/p}
 + \sup\limits_{h\in A}{\rm diam}(F_{k-1}(h), d)
\end{multline*}
and, similarly,
\begin{multline*}
d(\pi_{a2^{k/\alpha}}(f),h_{F_{k-1}(f)})^{p-1} + d(h_{F_{k-1}(f)},g)^{p-1}
\le
2\bigl(d(\pi_{a2^{k/\alpha}}(f),h_{F_{k-1}(f)}) + d(h_{F_{k-1}(f)},g)\bigr)^{p-1}
\\
\lesssim
\bigl(\sup\limits_{h\in A}\bigl(\varrho(\pi_{a2^{k/\alpha}}(h),h)\bigr)^{(p-1)/p} +
\sup\limits_{h\in A}{\rm diam}(F_{k-1}(h), d)\bigr)^{p-1}.
\end{multline*}
The above bounds imply
\begin{multline*}
\inf_{v\in T_k}\varrho(f,v)
\lesssim
\sup\limits_{h\in A}\varrho(h, \pi_{a2^{k/\alpha}}(h))
+
e_{k-1}(A^{F_{k-1}(f)}_{a2^{k/\alpha}}, \|\cdot\|_{h_{F_{k-1}(f)}})
+ \sup\limits_{h\in A}\bigl({\rm diam}(F_{k-1}(h), d)\bigr)^p.
\end{multline*}
We now apply Lemma \ref{Lem-A} 
to estimate the entropy number $e_{k-1}(A^{F_{k-1}(f)}_{a2^{k/\alpha}}, \|\cdot\|_{h_{F_{k-1}(f)}})$:
\begin{multline*}
e_{k-1}(A^{F_{k-1}(f)}_{a2^{k/\alpha}}, \|\cdot\|_{h_{F_{k-1}(f)}})
\\
\lesssim
\Bigl(\frac{a2^{k/\alpha}}{\eta}\Bigr)^{1/q}\Bigl({\rm diam}(A, \varrho)
+ \sup\limits_{h\in A}\varrho(h, \pi_{a2^{k/\alpha}}(h))\Bigr)^{1/q}
e_{k-1}(D, \|\cdot\|_{h_{F_{k-1}(f)}}),
\end{multline*}
where it is important that the entropy numbers are calculated
with respect to a norm.
Using the estimate $x^{1/q}y\le bx + b^{-1/(q-1)}y^{q/(q-1)}$ we get
\begin{multline*}
e_{k-1}(A^{F_{k-1}(f)}_{a2^{k/\alpha}}, \|\cdot\|_{h_{F_{k-1}(f)}})
\lesssim
b\, {\rm diam}(A, \varrho)
+ b\sup\limits_{h\in A}\varrho(h, \pi_{a2^{k/\alpha}}(h))\\
+
\Bigl(\frac{a2^{k/\alpha}}{b\eta}\Bigr)^{1/(q-1)}
\bigl(e_{k-1}(D, \|\cdot\|_{h_{F_{k-1}(f)}})\bigr)^{q/(q-1)}.
\end{multline*}
Therefore,
\begin{multline*}
\inf_{v\in T_k}\varrho(f,v)
\lesssim
b\, {\rm diam}(A, \varrho)+
(b+1)\sup\limits_{h\in A}\varrho(h, \pi_{a2^{k/\alpha}}(h))
\\
+
\sup\limits_{h\in A}\bigl({\rm diam}(F_{k-1}(h), d)\bigr)^p
+
\Bigl(\frac{a2^{k/\alpha}}{b\eta}\Bigr)^{1/(q-1)}
\sup\limits_{h\in A}\bigl(e_{k-1}(D, \|\cdot\|_h)\bigr)^{q/(q-1)},
\end{multline*}
which completes the proof of the lemma.
\end{proof}

{\bf Proof of Theorem \ref{T-VH}}

Let $s_k$ be as in Lemma \ref{Lem-B} for $k\ge1$ and let $s_0(f):={\rm diam}(B, \varrho)$,
then by Theorem \ref{Contraction}, one has
$$
\gamma_{\alpha,1}(B,\varrho)\le C(\alpha)\Bigl(b\, \gamma_{\alpha,1}(B,\varrho) +
\sup_{f\in B}\sum\limits_{k\ge0}2^{k/\alpha}s_k(f)\Bigr)
$$
which, in our case, provides the bound
\begin{multline*}
\gamma_{\alpha,1}(B,\varrho)
\lesssim
b\, \gamma_{\alpha,1}(B,\varrho) + {\rm diam}(B,\varrho)
+
(b+1)\sup\limits_{f\in B}\sum\limits_{k\ge1}2^{k/\alpha}\varrho(f, \pi_{a2^{k/\alpha}}(f))
\\+
\Bigl(\frac{a}{b\eta}\Bigr)^{1/(q-1)} \sup\limits_{f\in B}\sum\limits_{k\ge1}
\bigl(2^{k/\alpha}e_{k-1}(D, \|\cdot\|_f)\bigr)^{q/(q-1)}
+
\sup\limits_{f\in B}\sum\limits_{k\ge1}2^{k/\alpha}\bigl({\rm diam}(F_{k-1}(f), d)\bigr)^p
\end{multline*}
for any admissible sequence $(\mathcal{F}_k)$ of $B$.
Taking $b$ sufficiently small and applying Theorem \ref{Interpolation},
we get
\begin{multline*}
\gamma_{\alpha,1}(B,\varrho)
\lesssim
{\rm diam}(B,\varrho) +
a^{-1}\sup_{f\in B}\|f\|
+
\Bigl(\frac{a}{\eta}\Bigr)^{1/(q-1)} \sup\limits_{f\in B}
\sum\limits_{k\ge1}\bigl(2^{k/\alpha}e_{k-1}(D, \|\cdot\|_f)\bigr)^{q/(q-1)}
\\+
\sup\limits_{f\in B}\sum\limits_{k\ge1}2^{k/\alpha}\bigl({\rm diam}(F_{k-1}(h), d)\bigr)^p.
\end{multline*}
Since $B\subset D$ we get $\sup\limits_{f\in B}\|f\|\le \sup\limits_{f\in D}\|f\|=1$.
Taking infimum over all admissible sequences $(\mathcal{F}_k)$ of $B$ and
taking
$$
a=\Bigl(\eta^{-1/(q-1)} \sup\limits_{f\in B}
\sum\limits_{k\ge1}\bigl(2^{k/\alpha}e_{k-1}(D, \|\cdot\|_f)\bigr)^{q/(q-1)}\Bigr)^{-(q-1)/q},
$$
we obtain
\begin{multline*}
\gamma_{\alpha,1}(B,\varrho)
\lesssim
{\rm diam}(B,\varrho)
+
\eta^{-1/q}\Bigl(\sup\limits_{f\in B}
\sum\limits_{k\ge0}\bigl(2^{k/\alpha}e_k(B, \|\cdot\|_f)\bigr)^{q/(q-1)}\Bigr)^{(q-1)/q}
+
\gamma_{\alpha p,p}(B, d)^p
\end{multline*}
Since ${\rm diam}(B,\varrho)\le c_2{\rm diam}(B, \|\cdot\|_h)+c_2{\rm diam}(B, d)^p$, we get the claim of the theorem.

\section{Appendix B: the proof of Theorem \ref{T-Conv-2}}

We firstly formulate the desired statement.

{\bf Theorem \ref{T-Conv-2}.}{\it\
Let $B$ be a symmetric $q$-convex (with constant $\eta$)
body in some linear space $L$ and
let $\|\cdot\|$ be a norm on $L$.
Then
for any $\alpha>0$ and for any $p\in[1,q)$
there is a number $C(\alpha,p,q)>0$
such that
$$
\gamma_{\alpha,p}(B, \|\cdot\|)\le
C(\alpha,p,q)\eta^{-p/q}\Bigl(\sum_{k\ge0}(2^{k/\alpha}e_k(B, \|\cdot\|))^{\frac{pq}{q-p}}
\Bigr)^{\frac{q-p}{pq}}.
$$
}

The set $B$ is a unit ball of some $q$-convex (with constant $\eta$) norm $\|\cdot\|_B$.
Let 
$$K(t,f):=\inf\limits_{g\in L}(\|g\|_B + t^p\|f-g\|^p)$$
and let $\pi_t(f)$ be any minimizer.

We need the following lemma from \cite{VH} (see Lemma 5.9 there).

\begin{lemma}\label{Lem-3} For every $a>0$ one has
$$
\sup\limits_{f\in B}\sum_{k\ge0}\bigl(2^{k/\alpha}\|f-\pi_{a2^{k/\alpha}}(f)\|\bigr)^p
\le c(\alpha) a^{-p}.
$$
\end{lemma}

Similarly to the proof of Lemma \ref{Lem-A}, one can obtain the following lemma.

\begin{lemma}\label{Lem-C}
For any $t>0$ and for any $A\subset B$ one has
$$
{\rm diam}(A_t, \|\cdot\|_B)\le
c(p,q)\Bigl(\frac{t}{\eta}\Bigr)^{p/q}\Bigl({\rm diam}(A, \|\cdot\|) +
\sup\limits_{h\in A}\|h - \pi_t(h)\|\Bigr)^{p/q}
$$
where
$A_t:=\{\pi_t(h)\colon h\in A\}$.
\end{lemma}

{\bf Proof of Theorem \ref{T-Conv-2}}

From Lemma \ref{Lem-C}, for any $b>0$ we get the bound
\begin{multline*}
e_k(A_t, \|\cdot\|)
\le c(p,q)\Bigl(\frac{t}{\eta}\Bigr)^{p/q}\Bigl({\rm diam}(A, \|\cdot\|) +
\sup\limits_{h\in A}\|h - \pi_t(h)\|\Bigr)^{p/q}\!e_k(B,\|\cdot\|)
\\
\le
c(p,q)\bigl(b\, {\rm diam}(A, d) +
b\, \sup\limits_{h\in A}d(h, \pi_t(h)) + \Bigl(\frac{t}{b\eta}\Bigr)^{\frac{p}{q-p}}
e_k(B, \|\cdot\|)^{\frac{q}{q-p}}\bigr)
\end{multline*}
and
\begin{multline*}
e_k(A,\|\cdot\|)\le
c(p,q)\bigl(b\, {\rm diam}(A, \|\cdot\|) +
(b+1)\sup\limits_{h\in A}\|h - \pi_t(h)\|
+ \Bigl(\frac{t}{b\eta}\Bigr)^{\frac{p}{q-p}}e_k(B, \|\cdot\|)^{\frac{q}{q-p}}\bigr).
\end{multline*}
Taking $t=a2^{k/\alpha}$ and applying Theorem \ref{Contraction} we get
\begin{multline*}
\gamma_{\alpha,p}(B,d)\le c(\alpha,p,q)\Bigl(b\gamma_{\alpha,p} +
(b+1)\bigl[\sup\limits_{h\in B}\sum_{k\ge0}\bigl(2^{k/\alpha}
\|h - \pi_{a2^{k/\alpha}}(h)\|\bigr)^p\bigr]^{1/p}
\\+ \Bigl(\frac{a}{b\eta}\Bigr)^{\frac{p}{q-p}}
\bigr[\sum_{k\ge0}\bigl(2^{(1+\frac{p}{q-p})\frac{k}{\alpha}}
e_k(B, \|\cdot\|)^{\frac{q}{q-p}}\bigr)^p\bigr]^{1/p}\Bigr).
\end{multline*}
Taking $b$ sufficiently small and applying Lemma \ref{Lem-3} we get
$$
\gamma_{\alpha,p}(B,d)\le
c(\alpha,p,q)
\bigl(a^{-1}+ \Bigl(\frac{a}{\eta}\Bigr)^{\frac{p}{q-p}}
\Bigr[\sum_{k\ge0}\bigl(2^{k/\alpha}e_k(B, \|\cdot\|)\bigr)^{\frac{pq}{q-p}}\Bigr]^{1/p}\bigr).
$$
Optimizing over $a>0$ we get the desired bound.

\section{Appendix C: the proof of Lemma \ref{ent-est-2}}

We again point out that the proof follows the ideas of the proof of
\cite[Proposition 16.8.6]{Tal}.
Firstly, we recall the desired statement.

{\bf Lemma \ref{ent-est-2}.}{\it
Let $p\in(1,2)$ and let $\mu$ be a probability Borel measure
on a compact set $\Omega$.
There is a constant $C:=C(p)$ such that,
if $L$ is an $N$-dimensional subspace of~$L^p(\mu)\cap C(\Omega)$
such that
$$
\|f\|_\infty\le M\|f\|_2\quad \forall f\in L
$$
for some number $M\ge 2$,
then
for any fixed set of $m$ points $X=\{X_1, \ldots, X_m\}$
one has
$$
e_k(B_p(L), \|\cdot\|_{\infty,X})
\le C[\log m]^{1/2} [\log M]^{\frac{1}{p}-\frac{1}{2}}M^{2/p}
2^{-k/p},
$$
where $B_p(L)=\{f\in L\colon \|f\|_p\le1\}$.
}

{\bf Proof of Lemma \ref{ent-est-2}.}
Since $\|f\|_\infty^2\le M^2 \|f\|_2^2\le\|f\|_\infty^{2-p}\|f\|_p^p$,
we have $\|f\|_\infty\le M^{2/p}\|f\|_p$ for any element $f\in L$ implying
that for
any fixed set of points $X=\{X_1,\ldots, X_m\}$ one has
$e_0(B_p(L), \|\cdot\|_{\infty, X})\le C_0 M^{2/p}$.
We further use the following known property (see \cite[Lemma 16.8.9]{Tal}) of the entropy numbers:
$$
e_{k+1}(B_p(L),\|\cdot\|_{\infty, X})\le
2e_k(B_p(L), \|\cdot\|_2)\, e_k(B_2(L),\|\cdot\|_{\infty,X}),
$$
where $B_2(L):=\{f\in L\colon \|f\|_2\le1\}$.
We will also use the following classical dual Sudakov bound for the
entropy numbers of an Euclidean ball with respect to some norm $\|\cdot\|$:
$$
e_k(B_2(L),\|\cdot\|)\le c\, 2^{-k/2}\mathbb{E}_g\bigl\|\sum_{k=1}^{N}g_ku_k\bigr\|.
$$
Here $g=(g_1,\ldots, g_N)$ is the standard Gaussian random vector
and $\{u_1,\ldots, u_N\}$ is any orthonormal basis in $L$.
By this bound,
$$
e_k(B_2(L),\|\cdot\|_{\infty,X})
\le
c\, 2^{-k/2}\mathbb{E}_g\bigl\|\sum_{k=1}^{N}g_ku_k\bigr\|_{\infty, X},
$$
where $c$ is a numerical constant.
We now note that
\begin{multline*}
\mathbb{E}_g\bigl\|\sum_{k=1}^{N}g_ku_k\bigr\|_{\infty, X}
=
\mathbb{E}_g\max\limits_{1\le j\le m}\bigl|\sum_{k=1}^{N}g_ku_k(X_j)\bigr|
\\
\le c_1 \max\limits_{1\le j\le m}\bigl(\sum\limits_{k=1}^N|u_k(X_j)|^2\bigr)^{1/2}[\log m]^{1/2}
\le  c_1 M[\log m]^{1/2}
\end{multline*}
where we have used the known bound for the expectation
of the maximum of  Gaussian random variables (see \cite[Proposition 2.4.6]{Tal}).
Thus,
$$
e_k(B_2(L),\|\cdot\|_{\infty,X})
\le
c_2 M2^{-k/2}[\log m]^{1/2}.
$$
For any $r>1$, we also have
$$
e_k(B_2(L), \|\cdot\|_r)
\le
c\,  2^{-k/2}\mathbb{E}_g\bigl\|\sum_{k=1}^{N}g_ku_k\bigr\|_r.
$$
Note, that
\begin{multline*}
\mathbb{E}_g\bigl\|\sum_{k=1}^{N}g_ku_k\bigr\|_r
\le
\bigl(\mathbb{E}_X\mathbb{E}_g\bigl|\sum_{k=1}^{N}g_ku_k(X)\bigr|^r\bigr)^{1/r}
\le c_3\sqrt{r} \Bigl(\mathbb{E}_X\bigl(\sum_{k=1}^{N}|u_k(X)|^2\bigr)^{r/2}\Bigr)^{1/r}
\le
c_3M\sqrt{r}.
\end{multline*}
For a fixed $r>2$ we now proceed similar to the proof of \cite[Lemma 16.8.8]{Tal}.
Take any $R>r$ and let $\theta\in(0,1)$
be such that $\frac{1}{r} = \frac{1-\theta}{2} + \frac{\theta}{R}$.
Then one has $\|f\|_r\le \|f\|_2^{1-\theta}\|f\|_R^\theta$ and
$$
e_k(B_2(L),\|\cdot\|_r)\le 2e_k(B_2(L),\|\cdot\|_R)^\theta\le
c_4[2^{-k}RM^2]^{\theta/2}
=
c_4[2^{-k}RM^2]^{\frac{1}{2}-\frac{1}{r} + \frac{\theta}{R}}.
$$
Thus, since $M\ge1$ one has
$$
[2^{-k}M^2]^{\theta/R}\le M^{2/R},\quad R^{\theta/R}\le 2.
$$
Taking $R=2r\log M$, we get
$$
e_k(B_2(L),\|\cdot\|_r)
\le
c_5 r^{\frac{1}{2}-\frac{1}{r}} [2^{-k}M^2\log M]^{\frac{1}{2}-\frac{1}{r}}.
$$
By \cite[Lemma 16.8.10]{Tal}, we get
$$
e_k(B_{r'}(L),\|\cdot\|_2)\le c_6r^{\frac{1}{2}-\frac{1}{r}}
[2^{-k}M^2\log M]^{\frac{1}{2}-\frac{1}{r}}.
$$
Taking $r=p'$, we get
$$
e_{k+1}(B_p(L),\|\cdot\|_{\infty, X})
\le
c_7(1-1/p)^{\frac{1}{2}-\frac{1}{p}}[\log m]^{1/2} [\log M]^{\frac{1}{p}-\frac{1}{2}}
  M^{2/p}2^{-k/p}.
$$
The lemma is proved.

\vskip .1in

{\bf Acknowledgments.}

The author would like to thank Professor V.N. Temlyakov and Professor B.S. Kashin
for helpful and stimulating discussions.

The author is a Young
Russian Mathematics award winner and would like to thank its sponsors and jury.

The work was supported by the Russian Federation
Government Grant No.~14.W03.31.0031.

\end{document}